\documentclass{IOS-Book-Article}     
\usepackage[T1]{fontenc}
\usepackage{amsmath,amsthm}
\usepackage{amssymb}
\usepackage{epsfig}
\usepackage{graphicx}

\input xy
\xyoption{all}

\def\lra{\longrightarrow}

\def\pic{{\varphi}}

\def\End{{\rm End}}

\def\gm{{\mathbb G}_m}
\def\Ker{{\rm Ker}}
\def\Fr{{\rm Fr}}
\def\PGL{{\mathsf  P \mathsf G\mathsf L}}
\def\GL{{\mathsf G \mathsf L}}
\def\Stab{{\rm Stab}}

\def\dv{{\rm div}}
\def\lra{{\longrightarrow}}
\def\cM{{\mathcal M}}
\def\Hom{{\rm Hom}}
\def\Ext{{\rm Ext}}

\def\FS{{\mathcal F\mathcal S}}

\newcommand{\la}{\lambda}

\newcommand{\C}{\Bbb C}
\newcommand{\Q}{\Bbb Q}
\newcommand{\Z}{\Bbb Z}
\def\cT{{\mathcal T}}

\def\Gal{{\rm Gal}}

\def\ra{\rightarrow}

\def\C{{\mathbb C}}
\def\F{{\mathbb F}}
\def\P{{\mathbb P}}
\def\Q{{\mathbb Q}}

\def\G{{\mathcal G}}
\def\I{{\mathcal I}}

\def\Z{{\mathbb Z}}
\def\C{{\mathbb C}}
\def\N{{\mathbb N}}

\def\Aut{{\rm Aut}}
\def\Pic{{\rm Pic}}
\def\Div{{\rm Div}}

\def\ord{{\rm ord}}
\def\codim{{\rm codim \,}}

\newtheorem{definition}{Definition}[section]
\newtheorem{conjecture}[definition]{Conjecture}
\newtheorem{theorem}[definition]{Theorem}

\newtheorem{proposition}[definition]{Proposition}
\newtheorem{corollary}[definition]{Corollary}

\theoremstyle{definition}
\newtheorem{remark}[definition]{Remark}
\newtheorem{lemma}[definition]{Lemma}

\begin{document}
\pagestyle{plain}

\begin{frontmatter}          
%
\title{A Torelli theorem for curves over finite fields}
\runningtitle{A Torelli theorem for curves over finite fields}

\author{Fedor Bogomolov, Mikhail Korotiaev and Yuri Tschinkel}
\address{Courant Institute of Mathematical Sciences, N.Y.U.\\ 
251 Mercer str., New York, NY 10012, U.S.A.}

\

\centerline{\it To John Tate, with admiration}

\begin{abstract}
We study hyperbolic curves and their Jacobians over finite fields in the context of
anabelian geometry. 
\end{abstract}


\end{frontmatter}


\section{Introduction}
\label{sect:intro}

This paper is inspired by the foundational results and ideas of John Tate
in the theory of abelian varieties over finite fields. 
To this day, the depth of this theory has not been fully explored. 
Here we apply Tate's theorems to anabelian geometry of 
curves over finite fields.

\

Let $C$ be an irreducible smooth projective 
curve of genus $\mathsf g=\mathsf g(C)\ge 2$ defined over a field $k$
and let $C(k)$ be its set of $k$-rational points. 
When $k$ is the field of complex numbers, the 
complex torus
$$
{\rm H}^0(C(\C), \Omega^1_C)^{\vee}/{\rm H}_1(C(\C),\Z)
$$
is the set of complex points of an algebraic variety, the Jacobian variety $J$ of $C$.
Choosing a point $c_0\in C(\C)$ we get a map 
$$
\begin{array}{ccc} 
 C(\C) & \ra     &  J(\C) \\
 c     & \mapsto &  (\omega \mapsto \int_{\gamma} \omega),
\end{array}
$$ 
where $\omega\in \Omega^1_C$ is a global 
section of the sheaf of holomorphic differentials on $C$ and $\gamma$ 
is any path from $c_0$ to $c$. 
In a more algebraic interpretation, the abelian group 
$J(\C)$  is isomorphic to  $\Pic^0(C)/\C(C)^*$, the group of degree zero divisors on $C$ modulo  
principal divisors, and the map above is simply:
$$
\begin{array}{ccc} 
 C(\C) & \ra     &  J(\C) \\
 c     & \mapsto &  c-c_0.
\end{array}
$$
This construction can be carried out over any field $k$, provided $C(k)\neq \emptyset$ and 
also contains the basepoint $c_0$: 
by a fundamental result of Weil, the Jacobian $J$ is defined over the field of definition
of $C$, and the set-theoretic maps above arise from $k$-morphisms. 

For each $n\in \N$, we get maps 
$$
C^n(k)\stackrel{\sigma_n}{\longrightarrow} C^{(n)}(k) \stackrel{\varphi_n}{\longrightarrow} J(k)
$$
where $C^n$ is the $n$-th power and $C^{(n)}=C^n/\mathfrak S_n$ is the $n$-th
symmetric power of $C$, i.e., $C^{(n)}(k)$ is
the set of effective degree $n$ zero-cycles on $C$ which are defined over $k$.   
The map to the Jacobian assigns to a degree $n$ zero-cycle $c_1+\ldots +c_n\in  C^{(n)}(k)$ 
the degree 0 zero-cycle $(c_1+\ldots + c_n)-nc_0$.  
The maps $\varphi_n$ capture interesting geometric information. 
For example, $\varphi_{\mathsf g}$ is birational, 
which leads to an alternative definition of $J$ 
as the unique abelian variety birational to $C^{(\mathsf g)}$.  
The locus  $\Theta:=\varphi_{\mathsf g -1}(C^{(\mathsf g-1)})\subset J$
is an ample divisor, the theta-divisor. 
The classical Torelli theorem says that the pair $(J,\Theta)$, consisting 
of the Jacobian $J$ of $C$ and its polarization $\Theta$,  
determines $C$ up to isomorphism. This theorem holds over any 
field and is one of the main tools
in geometric and arithmetic investigations of algebraic curves, relating these
to much more symmetric objects - abelian varieties.

From now on, let $k_0$ be a finite field of characteristic $p$ and 
$k=\bar{k}_0$ an algebraic closure of $k_0$. 
Recall that $J(k)$ is a torsion abelian group, with $\ell$-primary part 
$$
J\{\ell\}\simeq (\Q_{\ell}/\Z_{\ell})^{2\mathsf g}, \,\,\,\,\text{ for } \ell\neq p.
$$ 
The description of $J\{p\}$ is slightly more complicated: 
there exists a nonnegative integer $\mathsf n \le \mathsf g$ such that
$J\{p\}\cong (\Q_p/\Z_p)^{\mathsf n}$. 
Nevertheless, 
as an abstract abelian group, $J(k)$  depends ``almost'' 
only on the genus $\mathsf g$ of $C$. 
The procyclic Galois group of $k/k_0$ acts on $J(k)$ and one can consider the 
Galois representation on the Tate-module:
$$
T_{\ell}(J):=\varprojlim J[\ell^n], \,\,\, \ell\neq p, 
$$
where $J[\ell^n]\subset J(k)$ is the subgroup of $\ell^n$-torsion points. 
Let $F_J$ be the characteristic polynomial of the Frobenius endomorphism on
$$
V_{\ell}(J):=T_{\ell}(J)\otimes \Q_{\ell}.
$$  
By a fundamental result of Tate, $F_J$ determines the Jacobian as an algebraic
variety, modulo isogenies:

\begin{theorem}[Tate~\cite{tate}]
\label{thm:tate}
Let $J,\tilde{J}$ be abelian varieties over $k_0$ and $F_J, F_{\tilde{J}} \in \Z[T]$ the
characteristic polynomials of the $k_0$-Frobenius endomorphism $\Fr$ 
acting on 
$V_{\ell}(J)$, resp. $V_{\ell}(\tilde{J})$. 
Then 
$$
\Hom(J,\tilde{J})\otimes \Z_{\ell} \stackrel{\sim}{\longrightarrow} 
\Hom_{\Z_{\ell}[\Fr]}(T_{\ell}(J), T_{\ell}(\tilde{J})).
$$ 
The abelian varieties $J$ and $\tilde{J}$ are isogenous if and only if $F_J=F_{\tilde{J}}$.
\end{theorem}

In particular, while the Galois-module structure of $J(k)$ distinguishes $J$
in a rather strong sense (but not up to isomorphism of abelian varieties, 
an example can be found in \cite{zarhin}, Section 12),
the group structure of $J(k)$ does not.

\

In this paper, we investigate a certain ``group-theoretic'' analog of the Torelli theorem
for curves over {\em finite} fields. This analog has a natural setting 
in the anabelian geometry of curves. Throughout, we work in characteristic $\ge 3$.

Let $J^1=J^1$ be the Jacobian of (degree  1 zero-cycles of) $C$ and 
\begin{equation}
\label{eqn:11}
\begin{array}{cccc}
j_1\,:\,  &C(k) & \hookrightarrow & J^1(k)\\
          &c & \mapsto & [c]
\end{array}
\end{equation}
the corresponding embedding. The Jacobian $J$ of degree 0 zero-cycles on $C$ acts 
on $J^1$, translating by points $c\in C(k)$.
Let $\tilde{C}$, resp.  $\tilde{J}$, be another smooth projective curve, resp. 
its Jacobian. We will say that 
$$
\phi\,:\, (C,J)\ra (\tilde{C},\tilde{J})
$$
is an isomorphism of pairs if there exists a diagram

\

\centerline{
\xymatrix{
J(k)\ar[d]_{\phi^0}   & J^1(k)\ar[d]_{\phi^1}     & \ar[l]_{j_1}\ar[d]_{\phi_s} C(k) \\
\tilde{J}(k) & \tilde{J}^1(k) & \ar[l]_{\tilde{j}_1}\tilde{C}(k)
}
}

\

where 
\begin{itemize}
\item $\phi^0$ is an isomorphism of abstract abelian groups;
\item $\phi^1$ is an isomorphism of homogeneous spaces, 
compatible with $\phi^0$;
\item the restriction $\phi_s\,:\, C(k)\ra \tilde{C}(k)$ of $\phi^1$  
is a bijection of sets.
\end{itemize}

\

Our main result is:

\begin{theorem}
\label{thm:main}
Let $k=\bar{\mathbb F}_p$, with $p\ge 3$, and let   
$C,\tilde{C}$ be smooth projective curves over $k$
of genus $\ge 2$, with Jacobians $J$, resp. $\tilde{J}$. 
Let 
$$
\phi \,:\, (C,J)\ra (\tilde{C},\tilde{J})
$$ 
be 
an isomorphism of pairs. Then  $J$ and $\tilde{J}$ are isogenous.
\end{theorem}

\begin{conjecture}
\label{conj:main}
Under the assumptions of Theorem~\ref{thm:main}, 
$C$ and $\tilde{C}$ are
isomorphic as algebraic varieties, modulo Frobenius twisting. 
\end{conjecture}

\

There are examples of
geometrically nonisomorphic curves over finite fields with 
isomorphic Jacobians, as (unpolarized) algebraic varieties over $k_0$.  
Pairs of such curves are given by 
$$
y^2=(x^3+1)(x^3-1) \,\, \text{ and }\,\, y^2=(x^3-1)(x^3-4)
$$
over $\F_{11}$ with Jacobian $E\times E$, for a supersingular
elliptic curve $E$, or
$$
y^2=x^5+x^3+x^2-x-1 \,\, \text{ and }\,\, y^2=x^5-x^3+x^2-x-1
$$
over $\F_3$, with a geometrically simple Jacobian (see \cite{oort}, \cite{howe}  
and the references therein). 

\

Theorem~\ref{thm:main} was motivated by 
Grothendieck's anabelian geometry. This is a 
program relating algebraic fundamental groups
of hyperbolic varieties over arithmetic fields 
to the underlying algebraic structure. 
One of the recent theorems in this direction is due to A. Tamagawa:
Let $\Pi$ be a {\em nonabelian} profinite group. 
Then there are at most
finitely many curves over $k=\bar{\mathbb F}_p$ with 
tame fundamental group  
isomorphic to $\Pi$ \cite{tamagawa}.
Tamagawa generalized previous results by Pop-Saidi \cite{pop} and 
Raynaud \cite{ray}, who proved similar statements under some technical
restrictions on curves. The main new ingredient in Tamagawa's proof is
a delicate geometric analysis of special loci in Jacobians.  

\

In the second part of this paper, we apply Theorem~\ref{thm:main} 
to a somewhat ortho\-gonal problem.
Namely, we focus on the prime to $p$ part of the
{\em abelianization} of the absolute Galois group
of the function field of the curve, together with the set of
valuation subgroups. Our main result (Theorem~\ref{thm:galois})  
is that for projective 
curves $C$ over $k$, of genus $\mathsf g(C)>4$, the pair
$(\G^a_K,\I)$, consisting of the abelianization of the Galois group
of $K=k(C)$ and the set $\I=\{I_{\nu}\}_\nu$ of procyclic subgroups $I_{\nu}\subset \G^a_K$
corresponding to nontrivial valuations of $K$, determines the isogeny class of 
the Jacobian of $C$.

\

Here is a road-map of the paper.
In Section~\ref{sect:curves}, included as a motivation for
Conjecture~\ref{conj:main},
we discuss certain subvarieties 
of moduli spaces of curves cut out by  
conditions on the order of zero-cycles of the form $c-c'$ on $C$ in the group $J(k)$
(i.e., images of Hurwitz schemes and their intersections). 
Typically, very few such conditions suffice to determine $C$, up to a {\em finite} choice.
In Section~\ref{sect:formal} we study the formal automorphism group $G_C$ of the pair $(C,J)$
and derive some of its basic properties. In Section~\ref{sect:groups} we
collect several group-theoretic results about profinite groups which we apply in
Section~\ref{sect:jab} to prove that  
any elements $\gamma,\tilde{\gamma}\in G_C$ have the property that 
some integral powers $\gamma^n, \tilde{\gamma}^{\tilde{n}}$ commute. 
We then prove that this holds for 
the Frobenius endomorphisms $\phi^0(\Fr)$ and $\tilde{\Fr}$, as 
elements in $\End_{\tilde{k}_0}(\tilde{J})$, 
whenever we have an isomorphism of pairs 
$\phi\,:\, (C,J)\ra (\tilde{C},\tilde{J})$. 
In Section~\ref{sect:detect} we apply the theory of integer-valued linear
recurrences as in \cite{cor-zan-inv} to obtain a sufficient condition for
isogeny of abelian varieties. 
In Section~\ref{sect:isog} we  
construct towers of degree 2 field extensions
$$
k_0\subset \ldots \subset k_n\ldots \subset k_{\infty}, \,\, \text{ resp. }\,\,
\tilde{k}_0\subset \ldots \subset \tilde{k}_n\ldots \subset \tilde{k}_{\infty},
$$ 
provide set-theoretic intrinsic definitions of
$J(k_n)$, resp. $\tilde{J}(\tilde{k}_n)$, and establish that 
$$
\phi^0(J(k_n))\subset \tilde{J}(\tilde{k}_n), \,\,\, \text{ for all } \,\, n.
$$
Combining Tate's theorem~\ref{thm:tate} with Theorem~\ref{thm:isoge}
we conclude that $J$ and $\tilde{J}$ are isogenous.
In Sections~\ref{sect:background} and \ref{sect:anabel} we discuss 
extensions and applications of Theorem~\ref{thm:main} to anabelian geometry.
In the Appendix we establish several geometric facts on abelian subvarieties 
in special loci in Jacobians, needed in Section~\ref{sect:anabel}.

\

\noindent
{\bf Acknowledgments:}
We are grateful to A. Venkatesh and U. Zannier for useful 
suggestions, and to B. Hassett and Yu. Zarhin for their comments. 
The first author was partially supported by NSF grant DMS-0701578. 
The third author was partially supported by
NSF grant DMS-0602333.

\section{Curves and their moduli}
\label{sect:curves}

Let $C$ be an irreducible smooth projective curve of genus $\mathsf g=\mathsf g(C)>1$ 
over a finite field
$k_0$ of characteristic $p$, with $C(k_0)\neq \emptyset$, and let $J=J_C$ be its Jacobian. 
The Jacobian of degree 1 zero-cycles $J^1$ 
is a principal homogeneous space for $J$. 
For $\ell$ a prime number let 
$$
J\{\ell\}:=\cup_{n\in \N}J[\ell^n]\subset J(k), \,\,\, \text{ resp. }\,\,\, T_{\ell}(J)=\varprojlim J[\ell^n]
$$ 
be the $\ell$-primary part of $J(k)$, resp.  
the Tate-module. 
For any set of primes $ S$, put 
$$
J\{S\}:=\bigoplus_{\ell\in S}J\{\ell\}\subset J(k).
$$
The order of $x\in J(k)$ will be denoted by $\ord(x)$.

\begin{lemma}
\label{lemm:a}
Let $C$ be a curve of genus $\mathsf g>1$. Let $J$ be its Jacobian
and $a\in J(k)$ be such that 
$$
a+C(k)\subset C(k)\subset J^1(k).
$$
Then $a=0$. 
\end{lemma}

\begin{proof}
Let $\langle a\rangle $ be the cyclic subgroup generated by $a$ and let $n$ be its order. 
The translation by $a$ gives an action of $\langle a\rangle$
on $J^1$ and a separable nonramified covering  $C\to C/\langle a\rangle$
of degree $n$.
The quotient $\bar{J}:=J/\langle a\rangle$ acts on the corresponding 
principal homogeneous space $\bar{J}^1=J^1/\langle a\rangle$. 
The image $\bar{C}$ of $C$ under the projection $J^1\ra \bar{J}^1$ has genus
$\bar{\mathsf g}=\mathsf g/n - 1/n+1<\mathsf g$, since $n\ge 2$ and 
$\mathsf g(C)\ge 2$. Hence the Jacobian of $\bar{C}$ is a proper abelian 
subvariety of $\bar{J}$, of dimension at most $\bar{\mathsf g}$.
It follows that the same holds for
its preimage $C$, contradicting the fact that $C$ generates $J$. 
 \end{proof}

\begin{definition}
\label{defn:matrix}
A ordered set $R_n=\{r_1,\ldots, r_n\}$ of integers $r_j>1$, with $p\nmid r_j$ for all $j$,
will be called an $n$-{\em string}.
Let $J$ be an abelian variety over $k$ and $X\subset J(k)$.
A ordered subset  $\{x_0, x_1,\ldots, x_n\}\subset X$ will be called an 
$R_n$-configuration on $X$ if $r_{j}=\ord(x_j-x_0)$, for $1\le j\le n$.  
\end{definition}

We will mostly consider the case when $X=C(k)\hookrightarrow J(k)$, 
where $C$ is a curve of genus $\mathsf g=\mathsf g(C)\ge 1$. 
Note that an isomorphism of pairs $\phi\,:\, (C,J)\ra (\tilde{C},\tilde{J})$ 
preserves all configurations, i.e., for all 
$n\in \mathbb N$, every $R_n$-configuration in $C(k)\subset J(k)$
is mapped to an $R_n$-configuration in $\tilde{C}(k)\subset \tilde{J}(k)$.

\begin{theorem}
\label{thm:2.3}
Let $C$ be a curve over $k=\bar{\mathbb F}_p$ of genus $\mathsf g > 1$.
Then there exists a string $R_n$, 
with $n<2\mathsf g$ such that
\begin{itemize}
\item $C(k)\subset J(k)$ contains an $R_n$-configuration,
\item there exist at most finitely many nonisomorphic curves
of genus $\mathsf g$ containing an $R_n$-configuration, modulo Frobenius twists. 
\end{itemize}
\end{theorem}

\begin{proof}
We write $\mathcal M_{\mathsf g,n}$ for the moduli space (stack) of genus $\mathsf g$ curves
with $n$-marked points. We start with the following

\begin{lemma} 
Every string $R_1$ defines an 
algebraic subvariety $\mathcal D_{R_1,\mathsf g}\subset \mathcal M_{\mathsf g,1}$ 
of dimension $2\mathsf g-1$ 
with a finite surjection onto a subvariety of $\mathcal M_{\mathsf g}$.
\end{lemma}

\begin{proof}
Moduli computation. 
The cycle $c_1-c_0$ of order $m$ prime to $p$ is
the same as a function $f$ on $C$ with divisor $m(c_1-c_0)$. 
It defines a separable cover $C\ra \P^1$ of degree $m$, which 
is completely ramified over two points: $0,\infty$. 
The variety of such covers is a Hurwitz scheme, it is defined over
$\mathbb F_p\subset k$. The genus computation gives an upper bound 
of $2\mathsf g$ for the number of additional ramification points. 
Since there are only finitely many covers of fixed degree with
fixed branch points in $\mathbb P^1$, the dimension of the 
corresponding Hurwitz scheme is bounded by $2\mathsf g-1$. 
\end{proof}

\begin{remark}
Over $\mathbb C$, this Hurwitz scheme is irreducible and has
dimension $2\mathsf g -1$. The generic point of this scheme
corresponds to a cover with simple additional ramification points
whose images are all distinct. 
\end{remark}

The subvariety of $\mathcal M_{\mathsf g}$ parametrizing 
curves with an $R_n$-configuration 
is contained in the intersection of varieties corresponding to 
configurations of order 1 built from appropriate subsets of $R_n$.   

We proceed by induction: Assume that $C$ contains an $R_n$-configuration
$\{c_0,\ldots, c_{n}\}\subset C(k)$ and let  $\mathcal D_{R_n}\subset M_{\mathsf g,1}$ 
be a union of irreducible subvarieties of dimension $2\mathsf g-n-1$, 
corresponding to curves with such a configuration, each having a finite map onto a subvariety of 
$\mathcal M_{\mathsf g}$. We will use Hrushovski's theorem \cite{hrush} 
for the Jacobian fibration of the universal curve
over the function field of each irreducible component $\mathcal D$ of $\mathcal D_{R_r}$: 
the number of points of finite
order (coprime to $p$) on a nonisotrivial curve embedded into an abelian
variety, over a function field of positive dimension, 
is bounded. In particular, there exists an $N_{n+1}$ such that:
\begin{enumerate}
\item there is a point $c_{n+1}$ with $c_{n+1}-c_0$ of order $N_{n+1}$,
\item the subvariety of $\mathcal D$ parametrizing curves with a torsion point of order $N_{r+1}$
is a proper subvariety.
\end{enumerate}
Iterating this, in at most $2\mathsf g-1$ steps we obtain a string $R$ 
and a zero-dimensional variety $\mathcal D$ such that $C(k)$ contains an $R$-configuration
which distinguishes $C$ from all but finitely many other genus  $\mathsf g$ curves over $k$.

Note that the presence of a given configuration is invariant under
Galois automorphisms. Since the subvarieties $\mathcal D\subset \mathcal M_{\mathsf g,1}$ 
are defined over $k_0$ we obtain in the end a 
subset of $k_0$-points in $\mathcal M_{\mathsf g,1}$.
 \end{proof}

Theorem~\ref{thm:2.3} is far from optimal.
If we assume that $c_0$ is defined over $k_0$ and 
$c_1$ over an extension of $k_0$ than 
the Galois conjugate of $c_1-c_0$ has the same order,
so that the corresponding point in the image of the Hurwitz scheme
in $\mathcal M_{\mathsf g,1}$ is singular and, moreover, nonnormal.

We have $\dim \mathcal M_{\mathsf g,1} = 3\mathsf g-2$ and $\codim \mathcal D_n =\mathsf g-1$. 
If the varieties $\mathcal D_n$ intersected with correct codimensions then
a 3-configuration would give a subvariety of dimension $1$ in $\mathcal M_{\mathsf g,1}$ and a
4-configuration - a zero-dimensional subvariety in $\mathcal M_{\mathsf g,1}$.

By generic local computations,  
the dimension of double and triple intersections of Hurwitz schemes corresponding to 
1-strings with coprime entries should be at most the dimension of
a transversal intersection of varieties of the same dimension.  
Thus we expect that a triple intersection has dimension 1, and that
quadruple intersections have dimension 0.

\begin{conjecture} 
\label{conj:four-points}
For any curve $C$ of genus $\mathsf g(C)\ge 2$ there exist a string 
$R_4$ and an $R_4$-configuration on $C$ such that
all curves $\tilde{C}$ with an $R_4$-configuration on $\tilde{C}$ 
realizing the $R_4$-string are Galois conjugated to $C$.  
Moreover, all such configurations on $C$ are also Galois conjugated.
\end{conjecture}

Clearly, this would imply a strong version of Conjecture~\ref{conj:main}.

\begin{remark} 
Consider $R_3 = \{2,3\}$.  
Transversality would give $3\mathsf g-2- (2\mathsf g-2)=\mathsf g$ in this case. 
However, the corresponding intersection is trivial.

Indeed, in general the set of solutions $nc_0 = nc$ is trivial for odd
$n\leq \mathsf g-1$, and a point $c_0$ invariant under a hyperelliptic involution.
For $n\leq \mathsf g-1$ and $n$ even the point $c$ is always invariant under a hyperelliptic
involution.
\end{remark}

In fact, we have a ``supertransversality'' for these Hurwitz schemes.

\begin{proposition}
Let $r_1,r_1'$ be coprime integers. Let 
$R_1=\{r_1\}$ and $R_{1}'=\{r_1'\}$ be the corresponding 1-strings
and $\mathcal Z:=\mathcal D_{R_1, \mathsf g}\cap  \mathcal D_{R_1', \mathsf g}\subset \mathcal M_{\mathsf g,1}$
the intersection of the associated Hurwitz schemes.
Then $\mathcal Z=\emptyset$, provided $\mathsf g \geq  (r_1-1)(r_1'-1)/2$.
\end{proposition}

\begin{proof}
The coprimality condition implies that the pair 
of functions $(f_{r_1},f_{r_1'})$, with divisors
$r_1(c_1-c_0)$, resp. $r_1'(c_1'-c_0)$, realizing the configuration,  
gives a map $C\ra \mathbb P^1\times \mathbb P^1$, birational
onto its image. The family of such curves in 
$\mathbb P^1\times \mathbb P^1$ is algebraic.
Hence $\mathsf g(C) \leq (r_1-1)(r_1'-1)/2$. 
A smooth curve in the family
has genus $\mathsf g=(C(C+K)/2) +1$  
(where $K=K_{\mathbb P^1\times \mathbb P^1}$ is the canonical class) 
which gives
$$
(r_1H + r_1'H')((r_1'-2)H + (r_1'-2)H')= (2r_1r_1' -2r_1-2r_1')/2 +1 =(r_1-1)(r_1'-1).
$$
The image of $C$ in $\mathbb P^1\times \mathbb P^1$ has a singularity
in the image of $c_0$, the same singularity as the rational curve
$(t^{r_1},t^{r_1'})$. This rational curve has the same homology class
as $C$ and has exactly two equivalent singularities, at
$(0,0)$ and at $(\infty,\infty)$.
Thus if $(r_1-1)(r_1'-1) - 2\delta (r_1,r_1') = 0$ then
the defect of the singularity is $\delta (r_1,r_1')=(r_1-1)(r_1'-1)/2$, 
which gives a lower bound for the defect for $C$.  
Hence $\mathsf g(C) \leq(r_1-1)(r_1'-1)/2 $.
\end{proof}

\begin{conjecture}
\label{conj:another}
Let  $f_{r_1},f_{r_1'},f_{r_1''}\in k(C)$  be functions as above
and $\lambda\in k^* \setminus \{1\}$. 
Assume that there are four points $c_0,c_1,c_2,c_3\in C(k)$ 
such that
$$
\begin{array}{rcl}
{\rm div}(f_{r_1})&  = & r_1(c_0-c_1)\\
{\rm div}(f_{r_1'})&  = & r_1'(c_0-c_2)\\
{\rm div}(f_{r_1}'')&  = & r_1''(c_0-c_3)
\end{array}
$$
and such that 
$$
f_{r_1}(c_2)=1 \text{ and } f_{r_1}(c_3)=\lambda.
$$
Then there are only finitely many curves $\tilde{C}$ with 
the same property.
\end{conjecture}

This would imply that the 3-point scheme intersection  
$$
\mathcal D_{R_1, \mathsf g}\cap  \mathcal D_{R_1', \mathsf g}\cap  
\mathcal D_{R_1'', \mathsf g}\subset \mathcal M_{\mathsf g,1}
$$
has dimension at most 1, and consequently the finiteness part of 
Conjecture~\ref{conj:four-points}.
We don't know whether or not this intersection
is irreducible. We would expect it at least for
sufficiently large coprime $r_1,r_1',r_1''$.

\section{Formal automorphisms}
\label{sect:formal}

Let $A$ be an abelian variety, $A^1$ a principal homogeneous space for $A$ and 
$X\subset A^1$ a subvariety not preserved by
the action of an abelian subvariety of $A$ of positive dimension. 

\begin{lemma}
The subgroup
$$
\Stab_X:=\{ a\in A(k)\,|\, a+X(k)\subset X(k) \}
$$
is finite.
\end{lemma}

\begin{proof}
See, e.g.,  \cite{abramovich}.
 \end{proof}

Let $\Aut(A)$ be the group of automorphisms of the
torsion abelian group $A(k)$. Note that $\Aut(A)$ is 
a profinite group since all of its orbits in $A(k)$ are finite.
In fact, we have
$$
\Aut(A)=\Aut(A)_p\times \prod_{\ell\neq p} \GL_{2\mathsf d}(\Z_{\ell}),
$$
with $\mathsf d=\dim A $ and $\Aut(A)_p=\GL_{n}(\Z_p)$, where $n$ is the rank of the 
\'etale $p$-subgroup of $A(k)$.

The group $\Aut(A)$ has a split affine extension 
\begin{equation}
\label{eqn:ext}
1\ra A(k)\ra \Aut(A)^{\rm aff} \ra \Aut(A)\ra 1
\end{equation}
which acts on $A^1(k)$. Since $A(k)$ is a discrete group, 
$\Aut(A)^{\rm aff}$ carries a natural topology. 
Let 
$$
G_X:=\{ \gamma \in \Aut(A)^{\rm aff}\,|\, \gamma(X(k))\subseteq X(k)\subset A^1(k)\}
$$
be the subgroup preserving $X(k)$. We call $G_X$ the group of 
automorphisms of the pair $(X,A)$.

\begin{lemma}
\label{lemm:closed}
The group $G_X$ is closed in the topological group $\Aut(A)^{\rm aff}$.
Its projection to $\Aut(A)$ has finite kernel.
\end{lemma}

\begin{proof}
Choose a point $x\in X(k)\subset A^1(k)$ and 
let $\Aut(A)_x\subset  \Aut(A)^{\rm aff}$ be a section of the projection in \eqref{eqn:ext}
of automorphisms fixing $x$. With this choice, we may assume that $X\subset A$. 
We claim that the subgroup 
$G_X\cap \Aut(A)_x$ has finite index in $G_X$, for all $x\in X$, 
and that its image in $\Aut(A)_x$ is closed in the profinite topology.

For each $\alpha\in A(k)$ the stabilizer 
$G_{\alpha}\subset G_X$ has finite index since the order of $\alpha$ is unchanged
under an automorphism of $X$. 
Note that  $G_X$ acts on the subgroup of $A(k)$ generated by zero-cycles
with support in $X(k)$. Consider the map 
$$
\begin{array}{rcc}
X\times X & \ra &  A \\
(x,x') & \mapsto & x-x',
\end{array}
$$
choose an $\alpha \in A(k)$
whose preimage in $X\times X$ is nonempty and of minimal dimension
and let $X_{\alpha}$ be the projection of this preimage to the first factor. 
Note that the stabilizer $G_{\alpha}$ preserves $X_{\alpha}(k)$ and 
that for dimension reasons
$\dim X_{\alpha}  < \dim X $. 
Thus we can assume that $G_{\alpha}$ maps to $G_{X_\alpha}$.
If $\dim X_{\alpha} = 0$ then 
$G_X\cap \Aut(A)_x$ has finite index in $G_{\alpha}$
(and hence in $G_X$) for any $x\in X_{\alpha}(k)$.
There are two possibilities:
\begin{enumerate}
\item  $X_{\alpha}$ is not preserved by the action of a nontrivial proper abelian
subvariety of $A$,    
\item $X_{\alpha}$ is preserved by the action of a nontrivial proper abelian subvariety
$B_{\alpha}\subset A$.
\end {enumerate}
In the first case we use induction on dimension:
if $\dim X_{\alpha} > 0$ then, by the inductive assumption, we can find an
$x\in X_{\alpha}(k)$ with $G_{X_\alpha}\cap \Aut(A)_x$ having
finite index in $G_{X_\alpha}$. The preimage of 
$G_{X_\alpha}\cap \Aut(A)_x$ in  $G_{\alpha}$ 
has also finite index in $G_X$ and we obtain the result.

In the second case, $G_{X_\alpha}$ contains $B_{\alpha}(k)$ as a subgroup
and $G_{X_\alpha}/B_{\alpha}(k)$ has
a subgroup of finite index  $G_{X_\alpha}/B_{\alpha}(k) \cap\Aut(A)_{x'}$,
by the inductive assumption. 
Thus $G_X$ contains a subgroup $G_{B_{\alpha}+x'}x$ of finite index. 
Consider other subvarieties $G_{X_\alpha'}$.
Then either there exists an $X_\alpha'$ which is not preserved
by the action of an abelian subvariety and we can apply the previous argument
to find a point $x\in X_\alpha'(k)$ or there is a  nontrivial $B$ such that all $X_{\alpha}$
are preserved by the action of $B$. Since the union of all $X_{\alpha}$ of minimal
nonzero dimension forms an open subset of $X$ we obtain that
$X$ is preserved by the action of $B$, contradicting our assumption.
Thus we can find at least one $x\in X(k)$ with $G_X\cap \Aut(A)_x$ of finite index
in $G_X$. Note that for other $x'\in X$ we have 
$$
G_X\cap \Aut(A)_x\cap \Aut(A)_{x'} = G_X\cap \Aut(A)_x\cap G_{x-x'}
$$ 
and hence it has finite index in $G_X$.
It follows that  $G_X\cap \Aut(A)_{x'}$ also has finite 
index in $G_X$, for any $x'\in X(k)$.

Thus the orbit $G_X\cdot x$ is finite. The stabilizer of $x$ in $G_X$ 
has finite index and lies in $\Aut(A)_x$. 
It remains to observe that the stabilizer is closed in $\Aut(X)_x$, using the 
same argument. 

Lemma~\ref{lemm:a} implies immediately that the projection has finite kernel.
 \end{proof}

\begin{remark}
The group $G_X$ always contains the procyclic subgroup $\hat{\Z}$ generated
by a Frobenius automorphism, and its extension by a finite group of
algebraic automorphisms of the pair $(X,A)$.  
\end{remark}

\begin{proposition}
\label{prop:support}
Let $A$ be an abelian variety of dimension $\mathsf d$ and  
$X\subset A^1$ a subvariety. Let $G_X$ be the group of 
automorphisms of the pair $(X,A)$.  
Let 
$$
\psi=\prod_{\ell\neq p}\psi_{\ell}\,: \, G_X\ra \prod_{\ell\neq p} \mathsf{GL}_{2\mathsf d} (\Z_{\ell})
$$
be the corresponding homomorphism. Then, for all $\gamma\in G_X, \gamma\neq 1$,
there are infinitely many $\ell$ such that  $\psi_{\ell}(\gamma)\neq 1$. 
\end{proposition}

\begin{proof}
Assume the contrary. Then there is a finite set of primes $S$ and a projection
$X\to A\{S\}$ such that $\gamma$ acts trivially on the fibers.
Let $y\in A\{S\}$ be such that $y\neq \gamma(y)$ and let
$X_y$ be the preimage of $y$ in $X$.
Then $\gamma(X_y) = X_{\gamma(y)}$. Moreover,
$X_{\gamma(y)} = X_y + \gamma(x) - x$.
Thus $X_{\gamma(y)} \subset X\cap (X+\gamma(x)-x)$.
On the other hand, $X\cap (X+ \gamma(x)-x)$ is a proper
subvariety of smaller dimension.

However, the number of points in $X$ and $X_y$
is the same for finite fields over which the points from
$A\{S\}$ are not defined, and the number of such fields is infinite.
Contradiction.
 \end{proof}

\begin{definition}
\label{dfn: formal-iso}
A homomorphism of abelian groups $\phi^0\,:\,  A(k)  \ra A(k)$ is
called a {\em formal isogeny} if it
arises from a sequence $\{\phi^0_i\}$ of {\em algebraic} 
isogenies $\phi^0_i\,:\, A\ra A$,
with the property that for all finite subgroups $G\subset A(k)$, 
there exists an $n(G)\in \N$
such that $\phi^0_i|_G = \phi^0_{i'}|G$, for all $i,i'\ge n(G)$.
\end{definition}

An example is a $\hat{\Z}^*$-power
of the Frobenius endomorphism $\Fr\in \End_{k_0}(A)$.

\begin{proposition}
Let  $(X,A)$ be a pair as in Proposition~\ref{prop:support} and 
let $\gamma\in G_X$ be an element which commutes with the Frobenius
action. Then $\gamma$ is a formal isogeny. 
\end{proposition}

\begin{proof}
By Tate's theorem~\ref{thm:tate},  $\End_{k_0}(A)\otimes \Z_{\ell}$
is equal to the centralizer of Frobenius in 
$\End(T_{\ell})$. By assumption, the reduction of $\psi_{\ell}(\gamma)$ 
is in this centralizer, modulo 
any finite power of $\ell$. Thus it is approximated by 
elements in $\End_{k_0}(A)$, on every finite subgroup of $A(k)$. 
\end{proof}

\section{Group-theoretic background}
\label{sect:groups}

In this section we collect some group-theoretic facts which will be needed
in the proof of Theorem~\ref{thm:second} - assuring that 
the Frobenius endomorphisms in $\End_k(J)=\End_k(\tilde{J})$ commute.

\begin{lemma} 
\label{lemm:commute-ell}
Let $\ell > n+1$ be a prime and
$G\subset \GL_{n}(\Z_{\ell})$ a closed subgroup with an abelian
$\ell$-Sylow subgroup.
Assume further that $G$ is generated by its $\ell$-Sylow subgroups. 
Then $G$ is abelian.
\end{lemma}

\begin{proof} 
Since $\ell > n+1$, the group $G$ does not contain elements of finite $\ell$-order.
Indeed, assume that $\gamma\in \GL_n(\Z_{\ell})$ has order $\ell$. 
Then it generates a subalgebra of the matrix algebra which contains a subfield
$\Q_{\ell}(\sqrt[\ell]{1})$, which has dimension $\ell-1$ over $\Q_{\ell}$, and 
has to embed into the natural representation space $\Q_{\ell}^n$. 
This implies that $\ell\le n+1$.

Consider the reduction homomorphism 
$$
\bar{\psi}_{\ell}\,:\, G\ra \GL_n(\Z/\ell).
$$ 
The preimage $G^0=\bar{\psi}^{-1}_{\ell}(1)$ of the identity 
in $\GL_n(\Z/\ell)$ is a normal pro-$\ell$ subgroup. 
In particular, $G_0$ is contained in every $\ell$-Sylow subgroup of $G$. 
Hence $G_0$ is abelian and torsion-free, i.e.,  $G_0\simeq \Z_{\ell}^{r}$, for some $r\in \N$.

\

{\em Step 1.} 
Since $G$ is generated by its $\ell$-Sylow subgroups, which are abelian, 
and $G_0$ is contained in all these subgroups, 
$G_0$ commutes with all elements of $G$.
Let $G_0'$ be the  $\ell$-component of the center of $G$. It is a torsion free
group isomorphic to $\Z_{\ell}^r$ and containing $G_0$ as a subgroup of finite index. 

Thus $G$ is a central extension
\begin{equation}
\label{eqn:cent}
1\ra G_0' \to G \to G'\ra 1
\end{equation}
where $G'$ is a finite group.

\

{\em Step 2.}
This central extension is defined by an element 
$\gamma \subset {\rm H}^2(G',\Z_{\ell})$, which has finite order 
since $G'$ is finite.
Hence $\gamma$ is the image of an element from ${\rm H}^1(G', \Z/\ell^m)$ 
for some $m\in \N$, under the 
Bockstein homomorphism. 
This means that the corresponding central extension is induced
from a homomorphism $G'\to (\Z/{\ell^m})^r$, i.e., 
we have a commutative diagram

\

\centerline{
\xymatrix{
            &                    & \tilde{G} \ar[d]\ar@{=}[r] & \tilde{G}   \ar[d]   & \\
  1\ar[r]   & G_0' \ar[d]\ar[r]  &  G \ar[r] \ar[d]_f         & G' \ar[r] \ar[d]     & 1\\
  1 \ar[r]  & \Z_{\ell}^r \ar[r] &  \Z_{\ell}^r \ar[d] \ar[r] & (\Z/\ell^m)^r\ar[d] \ar[r] & 1\\
            &                    & 1                          & 1                          &,   
}
}
\noindent
where $\tilde{G}=\Ker(f)$ is a finite normal subgroup of $G$.

\

{\em Step 3.}
Since $G$ has no $\ell$-torsion, $\tilde{G}$ 
has order coprime to $\ell$. It follows that $f$ admits a 
section $\sigma\, :\, \Z_{\ell}^r \to G$.

\

{\em Step 4.}
We claim  that $\Z_{\ell}^r$ acts trivially on
$\tilde{G}$ and that the extension
$$
1\ra \tilde{G}  \ra G   \stackrel{f}{\ra} \Z_{\ell}^r\ra  1
$$
splits. 

Let $g\in\GL_n(\Z_\ell)$ be an element of
infinite $\ell$-order (i.e., all reductions $\bar{\psi}_{\ell^m}(g)\in \GL_n(\Z/\ell^m)$ are
of nontrivial $\ell$-power order). Consider an element  $h\in G\subset \GL_n(\Z_\ell)$ of finite order. 
Assume that $g^{\ell}$ commutes with $h$. 
Then $g$ commutes with $h$.

We have $g=g_s g_u$ where
$g_s$ is semi-simple, $g_u$ is unipotent, and $g_s,g_u$ commute.
If an element $h\in \GL_n(\Z_{\ell})$ has finite order
and commutes with $g$ then it commutes with $g_s$ and $g_u$.
Note that $g^{\ell}_u = (g_u)^{\ell}$ and that they have the same
commutators. Thus we can assume $g=g_s$.
In this case the algebra $\Q_{\ell}[g]\subset {\rm Mat}_{n\times n}(\Q_{\ell})$ 
is a direct sum of fields $K_i^{(g)}$ (finite extensions of $\Q_{\ell}$).

The subalgebra in ${\rm Mat}_{n\times n}(\Q_{\ell})$ of elements 
commuting with $h$ is a direct sum of matrix algebras over division algebras with
centers $K_i^{(g)}$. We have a natural embedding of algebras
$\Q_{\ell}[g^{\ell}]\subseteq \Q_{\ell}[g]$.  If this embedding is an isomorphism then
$h$ commutes with $g$. 
Otherwise, there is a proper subfield $K_i^{(g^{\ell})}\subset K_i^{(g)}$, 
which does not contain the projection of $g$ to this component of the matrix algebra. 
Since it contains $g^{\ell}$, the corresponding extension has degree $\ell$, 
contradicting the assumption that $\ell>n$.  

\

{\em Step 5.} 
Since $G$ is generated by its $\ell$-Sylow subgroups and 
all elements of $\tilde{G}$ commute with $\Z_{\ell}^r$, it follows that
$\tilde{G} = 1$ and $G= \Z_\ell^r$.
\end{proof}

\begin{lemma} 
\label{lemm:group-use}
Let $\mathsf H'\to \mathsf H$ be a surjective homomorphism of finite
groups. Assume that we have an exact sequence
$$
1\ra \mathsf S_{\ell}\ra \mathsf H  \ra \mathsf C \ra 1
$$
where $\mathsf S_{\ell}$ is a nontrivial normal $\ell$-subgroup of $\mathsf H$,
$\mathsf C$ is a cyclic group whose order is a power of a prime number $\neq\ell$. 

Then there is an $\ell$-Sylow subgroup $\mathsf S_\ell'\subset \mathsf H'$
such that 
\begin{itemize}
\item $\mathsf S_{\ell}'$ surjects onto $\mathsf S_{\ell}$,
\item the normalizer $\mathsf N'$ of $\mathsf S_{\ell}'$ in $\mathsf H'$ surjects onto $\mathsf H$.
\end{itemize}
In particular, there exists an element $h'\in \mathsf N'$ 
of order coprime to $\ell$ which 
surjects onto a generator of $\mathsf C$.
\end{lemma}

\begin{proof} 
All $\ell$-Sylow subgroups of $\mathsf H'$ surject onto
$\mathsf S_\ell$. Hence they generate a proper normal subgroup $\mathsf S'\subset \mathsf H'$ which
surjects onto $\mathsf S_\ell$. Any $h'\in \mathsf H'$ acts 
(by conjugation) on the set $\mathcal S(\mathsf H')$ of $\ell$-Sylow subgroups of
$\mathsf H'$. 

Since $\mathsf S'$ acts transitively on $\mathcal S(\mathsf H')$ 
there exists an element $s'\in \mathsf S'$ such that
$h's'$ acts with a fixed point on $\mathcal S(\mathsf H')$. 
Let $\tilde{\mathsf S}'$ be an $\ell$-Sylow subgroup preserved by $h's'$. 
The normalizer $\mathsf N'$ of $\tilde{\mathsf S}'$ surjects onto $\mathsf H$. 
In particular, we can find an element  $\tilde{h}'$
contained in this normalizer, of order 
coprime to $\ell$, which is mapped to a generator of $\mathsf C$. 
\end{proof}

Let $\mathsf H$ be a finite group and $\ell,p$ two distinct primes. 
We say that $\mathsf H$ contains an $(\ell,p^m)$-extension
$
\{ \mathsf s\in \mathsf S_{\ell}, \mathsf n\in \mathsf N \}
$
if the following holds: 
\begin{itemize}
\item $\mathsf S_{\ell}\subset \mathsf H$ is an $\ell$-Sylow subgroup,
\item $\mathsf N\subset \mathsf H$ is a subgroup
containing $\mathsf S_{\ell}$ as a normal subgroup,
\item the quotient  $\mathsf C:=\mathsf N/\mathsf S_{\ell}$ 
is a cyclic group of order $p^{m+1}$,
\item $\mathsf n \in \mathsf N$ projects onto a generator of $\mathsf C$, 
\item $\mathsf s\in \mathsf S_{\ell}$ satisfies $
[\mathsf s,\mathsf n^{p^m}]\neq 1 \,\,\text{ in } \,\, \mathsf S_{\ell}$. 
\end{itemize}

\begin{corollary} 
\label{coro:group-use}
Let $\pi\, :\,  \mathsf H'\ra \mathsf H$ be a surjective homomorphism of finite groups.
Assume that $\mathsf H$ contains an $(\ell,p^m)$-extension
$\{ \mathsf s\in \mathsf S_{\ell}, \mathsf n\in \mathsf N \}$.
Then $\mathsf H'$ contains an $(\ell,p^{m'})$-extension
$\{ \mathsf s'\in \mathsf S'_{\ell}, \mathsf n'\in \mathsf N' \}$.
Moreover, 
\begin{itemize}
\item $m'\ge m$,
\item $\pi(\mathsf S_{\ell}')=\mathsf S_{\ell}$,
\item $\pi(\mathsf s') =\mathsf s$, 
\item $\pi(\mathsf n')=\mathsf n$. 
\end{itemize}
\end{corollary}

\begin{proof}
We start with the exact sequence
\begin{equation}
\label{eqn:another-exact}
1\ra \mathsf S_{\ell}\ra \mathsf N\ra \mathsf C\ra 1.
\end{equation}
The full preimage of $\mathsf N$ in $\mathsf H'$ contains an $\ell$-Sylow subgroup
$\mathsf S_{\ell}'$ of $\mathsf H'$. By Lemma~\ref{lemm:group-use}, the 
normalizer of $\mathsf S_{\ell}'$ in $\mathsf H'$ contains an element $\mathsf n'$
of order coprime to $\ell$ such that $\pi(\mathsf n')=n$, surjecting onto a generator
of $\mathsf C$. We may correct $\mathsf n'$ such that its order becomes a power of $p$. 
It is divisible by the order of $\mathsf C$, i.e., it equals $p^{m'+1}$, with $m'>m$. 
Let $\mathsf N'\subset \mathsf H'$ 
be the subgroup generated by $\mathsf S_{\ell}'$ and $\mathsf n'$. 
Take $\mathsf s'$ to be any element in the preimage $\pi^{-1}(\mathsf s)$.
Then $ \{ \mathsf s'\in \mathsf S'_{\ell}, \mathsf n'\in \mathsf N' \}$ is the required
$(\ell,p^{m'})$-extension. 
\end{proof}

Let $\mathsf G$ be a semi-simple linear algebraic group over $\Z$. 
We will use the following generalization of a theorem of Jordan:

\begin{theorem}
\label{thm:jordan}
Let $k_0$ be a field with $q=p^r$ elements. 
There exists an $n=n(\mathsf{G})\in \mathbb N$ such that
every subgroup $G\subset \mathsf{G}(k_0)$ with $p\nmid |G|$
contains an abelian normal subgroup $H\subset G$ with
$|G/H|\le  n$. 

Further, there exists an $\ell_0=\ell_0(\mathsf G)$ such that
for all primes $\ell'$ and all primes $\ell\ge \ell_0$ with $\ell\neq \ell'$, 
the $\ell$-Sylow subgroups of $\mathsf G(\Z/\ell')$ and 
$\mathsf G(\Z_{\ell'})$ are abelian. 
\end{theorem}

\begin{proof}
See \cite{brauer}, \cite{weis}. 
\end{proof}

\begin{proposition}
\label{prop:group}
Let $G$ be a profinite group. Let $S$ be an infinite set of primes. 
Assume that $G$ admits a continuous homomorphism 
$$
\psi=\prod_{\ell\in S} \psi_{\ell} \,:\, G\ra \prod_{\ell\in S} \mathsf G(\Z_{\ell}).
$$
Assume that for all $\gamma \in G$, $\gamma\neq 1$ one has
\begin{equation}
\label{eqn:ass}
\psi_{\ell}(\gamma)\neq 1 \in \mathsf G(\Z_{\ell})
\end{equation}
for infinitely many $\ell\in S$ (i.e., $\gamma$ has infinite support). 
Then 
\begin{enumerate}
\item the induced reduction map
$$
\bar{\psi}:=\prod_{\ell\in S} \bar{\psi}_{\ell} \,:\, G\ra \prod_{\ell\in S}  \mathsf G(\Z/\ell)
$$
is injective;
\item there exists an $\ell_0=\ell_0(\mathsf G)$ such that 
for all primes $\ell>\ell_0$ the $\ell$-Sylow subgroup of $G$ is 
abelian;
\item there exist a normal closed abelian subgroup $H\subset G$
and an $n=n(\mathsf G)$ such that $G/H$ has exponent bounded by $n$, 
i.e., the order of every element in $G/H$ is bounded by $n$. 
\end{enumerate}
\end{proposition}

\begin{proof}
Put 
$$
K_{\ell}:=\Ker(\mathsf G(\Z_{\ell})\ra \mathsf G(\Z/\ell)).
$$ 
We have an exact sequence
$$
1\ra \prod_{\ell\in S}K_{\ell} \ra 
\prod_{\ell\in S}\mathsf G(\Z_{\ell})\ra \prod_{\ell\in S}\mathsf G(\Z/\ell)\ra 1
$$
Our assumption implies that $\psi$ is injective, and we get an injection of 
the kernel of the reduction 
$\Ker(\bar{\psi})\hookrightarrow \prod_{\ell\in S}K_{\ell}$.
If we had a nontrivial 
$\gamma\in \Ker(\bar{\psi})$, its image $\psi(\gamma)$ would generate a nontrivial 
closed procyclic subgroup
isomorphic to $\prod_{\ell'\in  S'}\Z_{\ell'}\subset \prod_{\ell \in S}K_{\ell}$, 
for some infinite set $S'\subset S$. Thus, there would exist a nontrivial 
element $\gamma_{\ell'}\in \Ker(\bar{\psi})$
such that $\psi_{\ell}(\gamma_{\ell'})=1$ for all $\ell\neq \ell'$, contradicting our assumption. 
This proves the first claim.

The second claim follows by combining the injectivity of 
$$
\prod_{\ell'\in S\setminus \ell } \psi_{\ell}\,:\, G\ra \prod_{\ell'\in S\setminus \ell} G(\Z/\ell')
$$
with Theorem~\ref{thm:jordan}.

From now on, we assume that $\ell> \ell_0$ 
so that the $\ell$-Sylow subgroup of $G$ is abelian.

\begin{lemma}
\label{lemm:what}
There exists a constant $\kappa=\kappa(\mathsf G)$ such that 
for all $\ell>\ell_0$, there exists a normal abelian subgroup
$
\mathsf Z_{\ell}\subset \bar{\psi}_{\ell}(G)
$
of index 
$$
[\bar{\psi}_{\ell}(G):\mathsf Z_{\ell}]\le \kappa.
$$ 
\end{lemma}

\begin{proof}
If the image $\bar{\psi}_{\ell}(G)\subset \mathsf G(\Z/\ell)$ 
does not contain elements
of order $\ell$ we can directly apply Theorem~\ref{thm:jordan}
to conclude that $\bar{\psi}_{\ell}(G)$ contains a normal abelian subgroup
of index $\kappa(\mathsf G)=n(\mathsf G)$. 

We may now assume that the image does contain elements of order $\ell$. 
We claim that there do not exist
$\gamma,\gamma'\in G$ such that
\begin{itemize}
\item $\bar{\psi}_{\ell}(\gamma),\bar{\psi}_{\ell}(\gamma)$ have $\ell$-power order and 
\item $\bar{\psi}_{\ell}(\gamma),\bar{\psi}_{\ell}(\gamma)$ do not commute
in $\mathsf G(\Z/{\ell})$.
\end{itemize}
Otherwise, both $\psi_{\ell}(\gamma)$ and $\psi_{\ell}(\gamma')$ 
are contained in some $\ell$-Sylow subgroups of $\mathsf G(\Z_{\ell})$, 
which are both abelian, by the assumption $\ell>\ell_0$. 
By Lemma~\ref{lemm:commute-ell}, the subgroup of $\mathsf G(\Z_{\ell})$ generated 
by these $\ell$-Sylow subgroups is abelian, contradicting the second assumption.

If follows that all elements of $\ell$-power order in $\bar{\psi}_{\ell}(G)$ commute, 
so that the group $\bar{\mathsf S}_{\ell}$ generated by them is in fact the
$\ell$-Sylow subgroup of $\bar{\psi}_{\ell}(G)$. 
It is abelian and normal. Consider the exact sequence
\begin{equation}
\label{eqn:sect}
1\ra \bar{\mathsf S}_{\ell}\ra \bar{\psi}_{\ell}(G) \ra \mathsf U_{\ell} \ra 1
\end{equation}
where $\mathsf U_{\ell}:=\bar{\psi}_{\ell}(G)/\bar{\mathsf S}_{\ell}$. 
Since $\ell\nmid |\mathsf U_{\ell}|$ the sequence \eqref{eqn:sect} admits a section and 
there is an embedding 
$$
\mathsf U_{\ell}\hookrightarrow \bar{\psi}_{\ell}(G)\subset\mathsf G(\Z/\ell).
$$ 
We apply Theorem~\ref{thm:jordan} to conclude that $\mathsf U_{\ell}$ has an abelian normal subgroup
$\mathsf A_{\ell}\subset \mathsf U_{\ell}$ with $|\mathsf U_{\ell}/\mathsf A_{\ell}|\le n(\mathsf G)$. 
We have the diagram

\centerline{
\xymatrix{
1\ar[r] &\bar{\mathsf S}_{\ell}\ar[r]   &  \bar{\psi}_{\ell}(G)  \ar[r] & \mathsf U_{\ell}\ar[r] &  1\\   
1\ar[r] &\bar{\mathsf S}_{\ell} \ar@{=}[u]\ar[r]   & {\mathsf H}_{\ell}\ar[r] \ar[u] & \mathsf A_{\ell} \ar[u] \ar[r]& 1 \\
}
}

\

\noindent
where $\mathsf H_{\ell}$ is the full preimage of $\mathsf A_{\ell}$ in $\bar{\psi}_{\ell}(G)$. It is a normal 
subgroup of $\bar{\psi}_{\ell}(G)$ with 
$$
|\bar{\psi}_{\ell}(G)/ {\mathsf H}_{\ell}| =|\mathsf U_{\ell}/\mathsf A_{\ell}|\le n(\mathsf G).
$$
Let $\mathsf Z_{\ell}\subset \mathsf H_{\ell}$ be the centralizer of $\bar{\mathsf S}_{\ell}$, 
it is a normal abelian subgroup of $\mathsf H_{\ell}$.  
Lemma~\ref{lemm:what} follows if we show that the index $[\mathsf H_{\ell}:\mathsf Z_{\ell}]$
is bounded independently of $\ell$.

\

There is a section 
$$
\sigma \,:\,\mathsf A_{\ell} \to \bar{\psi}_{\ell}(G)\subset \mathsf G(\Z/\ell). 
$$ 
In particular, the finite abelian group 
$\mathsf A_{\ell}$ has at most $n:={\rm rank}(\mathsf G)$ generators.
Consider the conjugation action of $\mathsf A_{\ell}$ on $\bar{\mathsf S}_{\ell}$.
For $\mathsf a\in \mathsf A_{\ell}$ let $\mathsf C({\mathsf a})$ be the cyclic subgroup
generated by the image of $\mathsf a$ in the group of outer automorphisms of $\bar{\mathsf S}_{\ell}$.
It suffices to show that for each of the $\le n$ generators 
of $\mathsf A_{\ell}$ the order $|\mathsf C({\mathsf a})|$ 
is bounded independently of $\ell$ and $\mathsf a$.

Let $\mathsf C_p({\mathsf a})\subset \mathsf C({\mathsf a})$ 
be the $p$-Sylow cyclic subgroup, with $p^{m+1}=|\mathsf C_p({\mathsf a})|$.
We have an extension of abelian groups
\begin{equation}
\label{eqn:N}
1\ra \bar{\mathsf S}_{\ell}\ra \mathsf N_{\ell} \ra \mathsf C_p({\mathsf a})\ra 1.
\end{equation}
We claim that the length of the orbits of $\mathsf c\in C_p({\mathsf a})$ on
$\bar{\mathsf S}_{\ell}$ is universally bounded, 
provided that $q:=p^{m}$ and $\ell$ are sufficiently large.  
More precisely, we have:

\begin{lemma}
\label{lemm:okk}
There exists a constant $n'=n'(\mathsf G)$  
such that for all $\mathsf a\in \mathsf A_{\ell}$, all 
$\mathsf s\in \bar{\mathsf S}_{\ell}$ and all generators $\mathsf c$ of 
$\mathsf C_p({\mathsf a})$
the commutator
$$
[\mathsf s,\mathsf c^{q}]=1,
$$ 
provided $\ell, q:=p^m\ge n'$. 
\end{lemma}

\begin{proof}
We will argue by contradiction. 
We have 
$$
G=\varprojlim_i G_i, \,\,\, \text{ where } \,\,\,  G_i:= \prod_{j=1}^i \bar{\psi}_{\ell_j}(G),
$$
$\{\ell_1,\ell_2,\ldots\}$ is the set of primes, with $\ell_1=\ell$,  
and the maps $\pi_{i}\,:\, G_{i+1}\ra G_n$ are the natural projections.  
Assume that 
\begin{equation}
\label{eqn:contr}
[\mathsf s,\mathsf c^q]=\mathsf s'\neq 1 \,\,\text{ in } \,\, \bar{\mathsf S}_{\ell}.
\end{equation}

We apply Corollary~\ref{coro:group-use} inductively to conclude that 
each of the groups $G_i$ has an $(\ell, p^{m_i})$-extension
$$
\{ \mathsf s_i\in \mathsf S_{\ell,i}, \mathsf n_i\in \mathsf N_i\}.
$$
More precisely, 
there is a sequence of groups $\mathsf S_{\ell,i}\subset G_i$ and 
elements $\mathsf s_i, \mathsf n_i\in G_n$ with the following 
properties:
\begin{itemize}
\item $\mathsf S_{\ell,i}$ is an $\ell$-Sylow subgroup of $G_i$,
\item $\mathsf s_i\in\mathsf S_{\ell,i}$
\item $\mathsf n_i$ is in the normalizer of $\mathsf S_{\ell,i}$,
\item $\mathsf n_i$ has order $p^{m_i}$ with $m_i\ge m$,
\item $[\mathsf s_i, \mathsf n_i^{p^{m_i}}]\neq 1$, 
\item $\pi_{i}(\mathsf S_{\ell,i+1})=\mathsf S_{\ell,i}$, $\pi_{i+1}(\mathsf s_{i+1})=\mathsf s_i$, 
$\pi_{i}(\mathsf i_{i+1})=\mathsf n_i$, for all $i$. 
\end{itemize}
The corresponding limits 
$$
\gamma_{\mathsf s}=\varprojlim \mathsf s_i,\,\,\,\, \gamma_{\mathsf c}=\varprojlim \mathsf c_i\in G
$$
have infinite support and don't commute. Thus there exists a prime number $r> \ell, q$ 
(and $\ell_0(\mathsf G)$) 
such that 
$$
[\bar{\psi}_r(\gamma_{\mathsf s}), \bar{\psi}_{r}(\gamma_{\mathsf c})]\neq 1.
$$
Let $i$ be sufficiently large so that the prime $r$ is among the primes $\ell_1,\ldots, \ell_i$.
There is a natural projection 
$$
\bar{\psi}_r \,:\, G_i \ra \bar{\psi}_r(G)\subset \mathsf G(\Z/r).
$$
The $\ell$-Sylow subgroup $\mathsf S_{\ell,i}$ surjects onto the $\ell$-Sylow subgroup of 
$\bar{\psi}_r(G)$, which is abelian by Theorem~\ref{thm:jordan}.
Let 
$$
\bar{\mathsf N}_r\subset \bar{\psi}_r(G)\subset \mathsf G(\Z/r)
$$
be the {\em nonabelian} group generated by
$\bar{\psi}_r(\gamma_{\mathsf s})$ and $\bar{\psi}_r(\gamma_{\mathsf n})$, 
i.e., by $\bar{\psi}_r(\mathsf s_i)$ and $\bar{\psi}_r(\mathsf n_i)$.
It fits into an exact sequence
$$
1\ra \bar{\mathsf S}_{\ell, r}\ra \bar{\mathsf N}_r\ra \bar{\mathsf A}_r\ra 1,
$$
where $\bar{\mathsf S}_{\ell, r}$ is an abelian group of $\ell$-power order, 
$\bar{\mathsf A}_r$ a cyclic abelian group of order divisible by 
$p^{m+1}$, $p\neq \ell$.

Since $r\nmid  |\bar{\mathsf N}_r|$
we can apply Theorem~\ref{thm:jordan}: Any subgroup of $\mathsf G(\Z/r)$ 
of order coprime to $r$ has a normal abelian subgroup of index bounded by 
some constant $n(\mathsf G)$. 
However, any abelian normal subgroup of  
$\bar{\mathsf N}_r$
has index $\ge \min (\ell, q)$. 
We obtain a contradiction, when $\ell$ and $q$ are $\ge n(\mathsf G)$.
\end{proof}

This finishes the proof of Lemma~\ref{lemm:what}.
\end{proof}

We complete the proof of Proposition~\ref{prop:group}. Indeed, put
$$
\bar{H}:=\prod_{\ell\in S}\left( \bar{\psi}_{\ell}(G)\cap \mathsf Z_{\ell}\right) \subset 
\prod_{\ell\in S}\mathsf G(\Z/\ell).
$$
This is an closed abelian normal subgroup of $\psi(G)=\prod_{\ell\in S} \bar{\psi}_{\ell}(G)$.
Since $\psi$ is an injection, the preimage $H:=\psi^{-1}(\bar{H})$ is a closed abelian normal
subgroup of $G$. By Lemma~\ref{lemm:what},  
$[\bar{\psi}_{\ell}(G):\mathsf Z_{\ell}]\le \kappa$, for all $\ell$, 
the quotient $G/H$ has exponent bounded by $\kappa$.  
\end{proof}

\section{Curves and their Jacobians}
\label{sect:jab}

Let $C$ be a smooth projective curve of genus $\mathsf g\ge 2$ over a field $k$ 
and $J^n$ the Jacobian of degree $n$ zero-cycles, or alternatively, degree $n$ line
bundles on $C$, with the convention $J=J^0$. 
We have the diagram

\

\centerline{
\xymatrix{
C^n \ar[r]^{\!\!\!\!\!\!\!\!\!\sigma_n}   &  C^{(n)}\ar[d]^{\varphi_n}\\
                     {}                             & J^n. 
}}

\

\noindent
For any field $k_0$ we denote by $C^{(n)}(k_0)$ the set of 
$k_0$-points of the variety $C^{(n)}$, i.e., the set of
effective cycles $c_1+\ldots +c_n$ defined over $k_0$. 
We write $C(k_0)^{(n)}\subset C^{(n)}(k_0)$ for the subset 
of cycles $c_1+\ldots +c_n$ where {\em each} $c_i$ 
is defined over $k_0$. 
Put 
$$
W_n^r(C):=\{ [L]\in J^n \,|\, \dim {\rm H}^0(C,L)\ge r+1\}, \,\,\,\,\,\, W_n(C):=W_n^0(C).
$$

The map $\varphi_n$ is surjective for $n\ge\mathsf g$.  For $n=\mathsf g$ 
there is a divisor $D\subset J$ such that for all $x\in J(k)\setminus D(k)$, 
the fiber $\varphi^{-1}_n(x)$ consists of one point. 
For $n\ge 2\mathsf g-1$, the map $\varphi_n$ is a $\P^{n-\mathsf g}$-bundle.

We may fix a point $c_0\in C(k_0)$ 
and the embedding 
$$
\begin{array}{ccc}
C & \hookrightarrow & J\\
c &\mapsto & [c-c_0].
\end{array}
$$
This allows us to identify $J^n$ and $J$.

\

\begin{lemma}
\label{lemm:M}
Consider the exact sequence 
$$
1\ra \Z_{\ell}^n\ra \Q_{\ell}^n\ra (\Q_{\ell}/\Z_{\ell})^n\ra 1.
$$
Let $M\in \End(\Z_\ell)$ be an endomorphism which is contained in 
$\GL_n(\Q_{\ell})$. Consider the induced action
$$
M \,:\, (\Q_{\ell}/\Z_{\ell})^n\ra (\Q_{\ell}/\Z_{\ell})^n, 
$$
and let $\Ker(M)$ be the kernel of this map.
Then there is a canonical isomorphism
$$
\Z_{\ell}^n / M(\Z_{\ell}^n)\simeq \ker(M).
$$
\end{lemma}

\begin{proof}
Consider the module 
$M^{-1}(\Z_{\ell}^n)/\Z_{\ell}^n$ as a submodule of $(\Q_{\ell}/\Z_{\ell})^n$.
It is equal to $\Ker(M)$.
On the other hand, $M$ induces an isomorphism
$$
M^{-1}(\Z_{\ell}^n)/\Z_{\ell}^n \simeq \Z_{\ell}^n/M(\Z_{\ell}^n).
$$ 
\end{proof}

The following lemma will be used in Section~\ref{sect:isog}. 

\begin{lemma}
\label{lemm:ell-points}
Fix a prime number $\ell\neq p$ and assume that $J(k_0)\supset J[\ell]$. 
Let $k_1/k_0$ be a degree $\ell$-extension. 
Then 
\begin{itemize}
\item $\frac{1}{\ell} J(k_0)\subset J(k_1)$,
\item $J\{\ell\}\cap J(k_1) = \frac{1}{\ell} J(k_0)\cap J\{\ell\}$. 
\end{itemize} 
\end{lemma}

\begin{proof}
Let $\Fr$ be the $k_0$-Frobenius automorphism of $k$. Its action on the Galois-module
$V_{\ell}=V_{\ell}(J)$ is semi-simple, and decomposes 
$V_{\ell}=\oplus_i K_i$, where $K_i/\Q_{\ell}$ are
finite extensions. Note that the eigenvalues of the Frobenius on $V_{\ell}$
are not roots of 1 and hence $\Fr^n-1$ is always an invertible endomorphism of
$V_{\ell}$. The Tate-module $T_{\ell}:=T_{\ell}(J)$ contains a submodule $T_{\ell}'$
of finite index which is preserved by the Galois-action and decomposes as
$T_{\ell}'=\oplus_i \mathfrak o_i$, where 
$\mathfrak o_i\subset K_i$ are the rings of integers.
The maximal ideal of $\mathfrak o_i$ will be denoted by $\mathfrak m_i$.

The Frobenius acts on $\mathfrak o_i$ via multiplication with a 
unit $a_i\in \mathfrak o_i^*$.  
By assumption, it acts trivially on $J[\ell]$. 
Since every $\mathfrak o_i$ is isomorphic to a primitive $\Fr$-submodule
of $T_{\ell}$ we have $a_i=1 \mod \ell$, for all $i$, 
where $(\ell) \subset \mathfrak o_i$ 
is a power of the maximal ideal $\mathfrak m_i$.
We can write
$$
a_i=1+\ell^{m_i}\varpi^{h_i}\,\,\,\, \text{ for }\,\,\, m_i,h_i\in \N,
$$
where $\varpi_i$ is a generator of $\mathfrak m_i$, and the ideal 
$(\varpi^{h_i})\supset (\ell)$. It follows that
\begin{equation}
\label{eqn:frobe}
a_i^{\ell} =1 + \ell^{m_i+1} \varpi^{h_i} \mod \ell^{m_i+1} \varpi^{h_i+1}. 
\end{equation}

Consider the filtration
$$
(\Fr^{\ell}-{\rm Id})T_{\ell}'\subset (\Fr-{\rm Id})\subset T_{\ell}' 
$$
and a similar filtration
$$
(\Fr^{\ell}-{\rm Id})T_{\ell}\subset (\Fr-{\rm Id})\subset T_{\ell}. 
$$
Observe that
$$
|(\Fr-{\rm Id})T_{\ell}/(\Fr^{\ell}-{\rm Id})T_{\ell}| = 
|(\Fr-{\rm Id})T_{\ell}'/(\Fr^{\ell}-{\rm Id})T_{\ell}'| = \ell^{2\mathsf g},
$$
where the first equality follows from the fact that 
$T_{\ell}'\subset T_{\ell}$ is a submodule of finite index, and the second 
assertion follows from Equation~\ref{eqn:frobe}. 
From the exact sequence
$$
1\ra T_{\ell}\ra V_{\ell}\ra J\{\ell\} \ra 1
$$
we observe that $(\Fr-1)T_{\ell}/(\Fr^{\ell}-1)T_{\ell}$ is canonically 
isomorphic to $J\{\ell\}^{\Fr^{\ell}}/J\{\ell\}^{\Fr}$, by Lemma~\ref{lemm:M}. 
Note that
$J\{\ell\}^{\Fr^{\ell}}$ contains $\frac{1}{\ell}J\{\ell\}^{\Fr}$. Indeed,
by our assumption $J[\ell]\subset J\{\ell\}^{\Fr}$. 
If $x\in  J\{\ell\}^{\Fr}$ then $\ell \Fr(\frac{x}{\ell})= x$ and 
hence $\Fr(\frac{x}{\ell}) = \frac{x}{\ell} + x_0$, where $\ell x_0 =0$, so that
$x_0\in J[\ell]$. Iterating, we obtain that
$$
\Fr^{\ell}(\frac{x}{\ell}) = \frac{x}{\ell} + \ell x_0 = \frac{x}{\ell},
\,\,\, \text{ and  } \,\,\, \frac{x}{\ell} \in J\{\ell\}^{\Fr^{\ell}}.
$$
It follows that $J\{\ell\}^{\Fr^{\ell}}/J\{\ell\}^{\Fr}$ contains a subgroup isomorphic to 
$(\Z/\ell)^{2\mathsf g}$. This implies the second claim of the lemma. The first follows
via the same argument applied to arbitrary $x\in J(k_0)$.  
\end{proof}

\begin{lemma}
\label{lemm:yz}
For $n\ge 2\mathsf g-1$, a field $k_0$ such that 
$\# k_0$ is sufficiently large, 
any finite extension $k_1/k_0$ and any point $x\in J(k_1)$ there exist
points $y,z\in \P^{n-\mathsf g}(k_1)=\varphi_n^{-1}(x)$ such that the fiber
$\sigma^{-1}_n(y)$ is irreducible as a cycle over $k_1$ and $\sigma^{-1}_n(z)$ is
completely reducible over $k_1$. 
\end{lemma}

\begin{proof}
Follows from the equidistribution 
theorem \cite{katz}, Theorem 9.4.4.
\end{proof}

\begin{corollary}
\label{coro:existence}
We have
$$
J(k)=\cup_{\Phi\in \End_k(J)}\Phi(C(k)),
$$
and in fact  
\begin{equation}
J(k)=\cup_{n\in \mathbb N}\,\,\, n\cdot C(k).
\end{equation}
Moreover, there exists a finite extension $k_0'/k_0$ such that 
$C(k_1)$ generates $J(k_1)$, 
for all finite extensions $k_1/k_0'$.
\end{corollary}

\begin{proof}
The existence of a $y$ as in Lemma~\ref{lemm:yz} implies 
the first statement
(see \cite{bt-1}, Corollary 2.4, and \cite{bt-2}, Theorem 1). 
The second follows from the existence of $z$. 
\end{proof}

It will be useful to be able to bound indices of subgroups in $J(k_1)$ 
generated by fewer points from $C(k_1)$. 
Assume that $k_1/k_0$ is a finite extension with $\#k_1=q$ and 
such that $C(k_1)$ generates $J(k_1)$. 
Write 
$$
\#J(k_1)= q^{\mathsf g}(1+\Delta_q)  \,\,\, \text{ and  }\,\,\, \#C(k_1)=q(1+\delta_q) 
$$ 
We know that $\Delta_q,\delta_q = O(\frac{1}{\sqrt{q}})$, the implied constant 
depending only on the genus $\mathsf g(C)$. We may assume that $q$ is such that
\begin{equation}
\label{eqn:delta}
|\Delta_q|, |\delta_q| \le 1/2.
\end{equation}

\begin{lemma}
\label{lemm:pre-minus}
Let $D\subset C(k_1)$ be a subset of points such that
$$
D/\# C(k_1) \le \epsilon_q. 
$$
Let $H\subset J(k_1)$ be the subgroup generated by points in $C(k_1)\setminus D$. 
Then 
$$
I:=|J(k_1)/H|\le 
\frac{(2\mathsf g-1)!\mathsf g 2^{2\mathsf g-1}}{(1-\epsilon_q)^{2\mathsf g-1}}.  
$$
\end{lemma}

\begin{proof}
We have
$$
\# H = \frac{q^{\mathsf g}(1+\Delta_q)}{I}.
$$
Observe that
$$
\# (C(k_1) \setminus D)^{2\mathsf g-1} = 
\frac{1}{(2\mathsf g-1)!} q^{2\mathsf g-1} (1+\delta_q)^{2\mathsf g-1}(1- \epsilon_1)^{2\mathsf g-1}.
$$
On the other hand, $C^{(2\mathsf g-1)}\ra J^1$ is a split projective bundle
of relative dimension $\mathsf g-1$. This implies that
$$
\frac{1}{(2\mathsf g-1)!} 
q^{2\mathsf g-1} (1+\delta_q)^{2\mathsf g-1}(1- \epsilon_q)^{2\mathsf g-1}
\le 
\frac{q^{\mathsf g}(1+\Delta_q)}{I}\cdot \frac{q^{\mathsf g}-1}{q-1}.
$$ 
Using the bound \eqref{eqn:delta}, we obtain
$$
I < \frac{(2\mathsf g-1)!\mathsf g}{((1+\delta_q)(1- \epsilon_q))^{2\mathsf g-1}}.   
$$
\end{proof}

Recall that the Galois group 
$\Gamma:=\Gal(k/k_0)$ is isomorphic to $\hat{\Z}=\prod_{\ell}\Z_{\ell}$
and is topologically generated by the Frobenius automorphism $\Fr$. 
For a finite set of primes $S$ let
$k_{S}\subset k$ be the fixed field of 
$\Gamma_{S}:=\prod_{\ell\notin S}\Z_{\ell}$;
the Galois group of the (infinite) extension $k_{S}/k_0$ is 
$\prod_{\ell\in S}\Z_{\ell}$.
Note that $J\{S\}\subset J(k_{S})$ and that
$C(k_{S})\subset J(k_{S})$
is infinite.  We have a natural projection map
$$
\lambda_{S}\,:\, C(k)\ra J(k)\ra J\{S\},
$$
(depending on the choice of $c_0$).

\begin{theorem}
\label{thm:main2}
Let $S$ be a finite set of primes. Then
\begin{itemize}
\item the set $C(k)\cap J\{S\}$ is finite;
\item the map 
$\la_{S}\,:\, C(k_{S})\ra J\{S\}$
is surjective with infinite fibers.
\end{itemize}
\end{theorem}

\begin{proof}
The first statement is due to Boxall \cite{boxall}. The second was proved in 
\cite{bt-2}.
 \end{proof}

\begin{remark}
\label{rem:boxall}
Boxall's theorem can be proved using the following statement.
Let $\Gamma\simeq \Z_\ell\subset \GL_n(\Z_\ell)$ be an analytic semi-simple subgroup
such that for all $F\in \Gamma$ one has
$F-1\in \GL_n(\Q_\ell)$.
Consider the induced action of $\Gamma$ 
on the torsion group $(\Q_\ell/\Z_\ell)^n$.
Then for all $m\in \N$ there is an
$r\in \N$ such that for all $x$ with $\ord(x) >\ell^r$ the orbit of $x$ contains
a translation of $x$ by a cyclic subgroup of order $ > \ell^n$.
\end{remark}

\begin{remark}
\label{rem:genera}
Theorem~\ref{thm:main2} admits a generalization: 
Let $X\subset A$ be a proper subvariety of an abelian variety.  
If $S$ is a finite set of primes and if the intersection 
$Y:=X(k)\cap \prod_{\ell\in S} A\{\ell\}$ is infinite 
then
$$
Y\subset \left(\cup_{i\in I} x_i+A_i(k)\right) \subset X(k)\subset A(k),
$$
where $I$ is a finite set, $A_i\subset A$ are abelian subvarieties and $x_i\in A(k)$ \cite{boxall}.
\end{remark}

Note that for finite fields $k_0$ with $\#k_0$ 
sufficiently large, the
image of $C(k_0)^{(\mathsf g)}$ does not coincide with $J(k_0)$.
Indeed, the number of $\mathbb F_q$-points in $C(\mathbb F_q)^{(\mathsf g)}$ is
approximately equal to
$$
\frac{q^{\mathsf g}}{\mathsf g!}  < q^{\mathsf g}.
$$ 
On the other hand, among 
infinite extensions of $k'/k_0$ we can easily find 
some with $C(k')^{(g)}= J(k')$.

\begin{proposition}
Let $k_0$ be a finite field with algebraic closure $k$, $S$ the set of primes $\le \mathsf g$
and  $\Gamma_S=\prod_{\ell\notin S} \Z_{\ell} \subset \Gal(k/k_0)$. 
Put $k':=k^{\Gamma_S}$. Then
$$
C(k')^{(\mathsf g)}= J(k').
$$
\end{proposition}

\begin{proof} 
There exists a subvariety $Y\subset J$ of codimension $\ge 2$
such that for all  $x\in J(k)\setminus Y(k)$ 
there is
a unique representation 
$x=\sum_{i=1}^{\mathsf g} c_i$, with $c_i\in C(k)$,  
modulo permutations.

Assume that $x\in J(k')\setminus Y(k')$ and that 
its representation as a cycle contains at least one
$c_i\notin C(k')$. 
For any $\gamma\in \Gamma_S$
we have $x=\sum_{1}^{\mathsf g}\gamma(c_i)$.
If $\gamma\neq 1$, then 
the size of any nontrivial orbit of $\gamma$
is strictly greater than $\mathsf g$.
Thus there is more than one representation of $x$
as a sum of points in $C(k)$, modulo permutations within the cycle. 
Contradiction.

Assume that $x\subset Y(k')$. Consider the fibration $C^{(g)} \ra J$. The 
fiber over $x$ is the projective space $\P^r$, defined over $k'$, 
parametrizing all representations of $x$ as a sum of degree $\mathsf g$ zero-cycles.
There exists  
$(c_1,\ldots, c_{\mathsf g})\in C^{(g)}(k')$ with 
$\sum_{i=1}^{\mathsf g}c_i =x$. We are done if $c_i\in C(k')$, for all $i$.
Otherwise,  we can apply the argument above, observing that
$\Gamma_S$ preserves this cycle.
\end{proof}

\begin{lemma}  
\label{lemm:formall}
Let $J_\gamma(k)\subset J(k)$ be the subgroup of 
elements fixed by $\gamma\in G_C$. If $C$ is not hyperelliptic then 
$$
j_\gamma : C(k) \setminus C_\gamma(k) \to J(k)/J_\gamma(k)
$$ 
is an embedding of sets. If $C$ is hyperelliptic let 
$$
C[4]:= \{\, c\in C(k)\,|\,  c\in J[4]\, \text{ and } \, \gamma(c)=-c\, \}.
$$
Then
$$
j_\gamma : C(k) \setminus (C_\gamma(k)\cup C[4]) \to J(k)/J_\gamma(k)
$$ 
is an embedding of sets. 
\end{lemma}

\begin{proof} 
Assume there exist two points $c,c'\in C(k)$ with $\gamma(c)\neq c$ and 
$\gamma(c')\neq c'$ and such that $j_\gamma(c)=j_{\gamma}(c')$.
Then $\gamma(c)-\gamma(c') =c-c'$ and hence $\gamma(c) + c' = c+ \gamma(c')$.
The cycles $\gamma(c) + c', c + \gamma(c')$ consist of different points since
$c'\neq c, c'\neq \gamma(c')$, by assumption.
Thus $\gamma(c) + c'$ defines a hyperelliptic pencil and we have
proved the lemma for nonhyperelliptic curves.

In the hyperelliptic case assume that the pencil
consists of elements $c, -c$ (since the pencil is clearly
$\gamma$-invariant and belongs to $J_\gamma$). Thus $c' = -c$
and $\gamma$ acts as $-1$ on $c$. Note that $j_{\gamma}(c) = -j_{\gamma}(c)$
implies that $j_{\gamma}(2c)= 0$ and $2c \in J_\gamma(k)$. Then $2c = -2c$
implies that $4c = 0$. Thus in this case
a possible exceptional subset consists of points $c\neq c' =-c$
of order $4$ such that $\gamma(c) = - c$.
 \end{proof}

\begin{theorem}
\label{thm:gd}
The group of automorphisms $G_C$ satisfies conditions of Proposition~\ref{prop:group}.
\end{theorem}

\begin{proof}
By Lemma~\ref{lemm:closed}, there is an injective 
continuous homomorphism 
$$
\psi=\psi_{\ell}\,:\, G_C\ra \prod_{\ell} \GL_{2\mathsf g}(\Z_{\ell}).
$$ 
Moreover, for all nontrivial $\gamma\in G_C$ 
the image $\psi_{\ell}(\gamma)\neq 1$, for infinitely many $\ell$.
Otherwise, let $S$ be the finite set of primes such that $\psi_{\ell}(\gamma)$
is trivial for $\ell\notin S$. Then 
\begin{itemize}
\item $J\{S\}\ra J(k)/J_{\gamma}(k)$ is a surjection;
\item $C(k)\ra J(k)/J_{\gamma}(k)$ is finite outside
$0\in J(k)/J_{\gamma}(k)$, by Lemma~\ref{lemm:formall};
\item $C(k)\ra J\{S\}$ is a surjection with infinite fibers over every point, 
by Theorem~\ref{thm:main2}. 
\end{itemize}
Contradiction.
 \end{proof}

\begin{corollary}
\label{coro:commute}
For all $\gamma,\tilde{\gamma}\in G_C$ there exists an  $n\in \N$ such that
$\gamma^n$ and $\tilde{\gamma}^{n}$ commute. 
\end{corollary}

\begin{proof}
If suffices to combine Theorem~\ref{thm:gd} and Proposition~\ref{prop:group}. 
 \end{proof}

\begin{theorem}
\label{thm:second}
Let $\phi\,:\, (C,J)\ra (\tilde{C},\tilde{J})$ 
be an isomorphism of pairs.
Then there exists an $n\in \N$ such that 
$\Fr_C^n$ and $\phi^{-1}(\Fr_{\tilde{C}}^n)$ commute in $\End_k(J)$. 
\end{theorem}

\begin{proof}
Immediate from Theorem~\ref{thm:gd} and Corollary~\ref{coro:commute}.
\end{proof}

\begin{lemma}
Assume that $\Fr$ and $\tilde{\Fr}$ generate the same $\ell$-adic subgroup in $\GL_n(\Z_\ell)$.
Then there exist  $n,\tilde{n}\in \N$
such that
$$
\Fr^{n} = \tilde{\Fr}^{\tilde{n}}.
$$
\end{lemma}

\begin{proof}
The assumption implies that there exist an
$\alpha\in \Z_{\ell}^*$ and an $\tilde{n}\in \N$ such that
$$
\Fr^{\alpha} = \tilde{\Fr}^{\tilde{n}}.
$$
The same equality holds for the determinants.
However, the determinants are positive integer powers of $p$.
 \end{proof}

\section{Detecting isogenies}
\label{sect:detect}

In this section, we recall some facts from divisibility theory for linear recurrences, as
developed in \cite{cor-zan-inv}, and apply these to derive a sufficient condition 
for isogeny of abelian varieties. 

\

A function $F\,:\, \N\ra \C$  is called a {\em linear recurrence} if
there exist an $r\in \N$, and $a_i\in \C$, 
such that for all $n\in \N$ one has
$$
F(n+r) = \sum_{i=0}^{r-1} a_i F(n+i).
$$
There is a unique expression 
\begin{equation*}
\label{eqn:uniq}
F(n)=\sum_{i=1}^m f_i(n)\gamma_i^n,
\end{equation*}
where $f_i\in \C[x]$ are nonzero and $\gamma_i\in \C^*$.
The complex numbers $\gamma_i\in \C^*$ are called the roots of the
recurrence. Let $\Gamma$ be a torsion-free finitely-generated subgroup of
the multiplicative group $\C^*$.  
Then the ring of linear recurrences with roots in $\Gamma$ is isomorphic
to the unique factorization domain $\C[x,\Gamma]$ (see \cite[Lemma 2.1]{cor-zan-inv}); 
the element in $\C[x,\Gamma]$ corresponding to a linear recurrence $F$ will be denoted by 
the same letter.

We say that $\{F(n)\}_{n\in \N}$ is a {\em simple}
linear recurrence, if $\deg(f_i)=0$, for all $i$, i.e., $f_i$ are constants.

\begin{proposition}
\label{prop:cz2}
Let $\{F(n)\}_{n\in \N}$, $\{\tilde{F}(n)\}_{n\in \N}$ be simple linear recurrences
such that $F(n), \tilde{F}(\tilde{n})\neq 0$ for all $n,\tilde{n}\in \N$.
Assume that
\begin{enumerate}
\item The set of roots of $F$ and $\tilde{F}$ generates a torsion-free
subgroup of $\C^*$.
\item There is a finitely-generated subring $\mathfrak R\subset\C$
with $F(n)/\tilde{F}(n)\in \mathfrak R$, for infinitely many $n\in \N$.
\end{enumerate}
Then 
$$
\begin{array}{rcc}
G\,:\, \N & \ra & \C\\
n         & \mapsto &  F(n)/\tilde{F}(n)
\end{array}
$$ 
is a simple linear recurrence. 
\end{proposition}

\begin{proof}
The fact that $G$ is a linear recurrence is proved in \cite[p.~434]{cor-zan-inv}.
Enlarging $\Gamma$, if necessary, we obtain an identity 
$$
G\cdot \tilde{F} = F,
$$
in the ring $\C[x,\Gamma]$. Since $F,\tilde{F}$ are simple, i.e., in $\C[\Gamma]$, 
$G$ is also simple. 
\end{proof}

\begin{lemma}
\label{lemm:poly}
Let $\Gamma$ be a finitely-generated torsion-free abelian group of rank $r$ with a fixed basis
$\{ \gamma_1,\ldots, \gamma_r\}$. Let
$\C[\Gamma]$ be the corresponding algebra of Laurent polynomials, 
i.e., finite linear combinations of monomials $x^{\gamma}=\prod_{j=1}^r x_j^{g_j}$, 
where $\gamma=\sum_{i=1}^r g_i\gamma_i\in \Gamma$. 
Let $\gamma$ be a primitive element in $\Gamma$, i.e., $\gcd(g_1,\ldots, g_r)=1$.
Then, for each $\lambda\in\C^*$, the polynomial $x^{\gamma}-\lambda$ is irreducible in $\C[\Gamma]$, 
i.e., defines an irreducible hypersurface in the torus $(\C^*)^r$.

Let $\gamma,\gamma'\in \Gamma$ be arbitrary elements. The polynomials 
$x^{\gamma}-1$ and   $x^{\gamma'}-1$ are not coprime in $\C[\Gamma]$, i.e., the corresponding divisors
in $(\C^*)^r$ have common irreducible components, if and only if $\gamma,\gamma'$ generate a cyclic 
subgroup of $\Gamma$. 
\end{lemma}

\begin{proof}
The map defined by the monomial $x^{\gamma}\,:\, (\C^*)^r\ra \C^*$ has irreducible fibers, 
if and only if $\gamma$ is primitive. For other $\gamma$, put $m:=\gcd(g_1,\ldots, g_r)>1$ and 
$\gamma =m \bar{\gamma}$. 
Then $x^{\gamma}-1=\prod_{s=1}^m (x^{\bar{\gamma}}-\zeta_m^s)$, where $\zeta_m$ is 
a primitive $m$-th root of 1.
By the the first observation, the polynomials  
$x^{\bar{\gamma}}-\zeta_m^s$ are irreducible. 
To prove the last statement, note that coprimality of $x^{\gamma}-1$ and   $x^{\gamma'}-1$
is equivalent to coprimality of  $x^{\bar{\gamma}}-1$ and   $x^{\bar{\gamma}'}-1$, for
the corresponding primitivizations $\bar{\gamma}, \bar{\gamma}'$ of $\gamma,\gamma'$.
This coprimality is equivalent to $\bar{\gamma}\neq \pm \bar{\gamma}'$.    
\end{proof}

Let $A$ be an abelian variety
of dimension $\mathsf g$ defined over a finite field $k_1$ of characteristic $p$, 
and let $\{\alpha_j\}_{j=1,\ldots, 2\mathsf g}$ be the set of eigenvalues
of the corresponding Frobenius endomorphism $\Fr$ on the $\ell$-adic
cohomology, for $\ell\neq p$.   
Let $k_n/k_1$ be the unique extension of degree $n$. 
The sequence
\begin{equation}
\label{eqn:fn}
F(n):=\#A(k_n) = \prod_{j=1}^{2\mathsf g} (\alpha_j^n-1).
\end{equation}
is a simple linear recurrence.
Let $\Gamma$ be the multiplicative subgroup of $\C^*$ generated by 
$\{\alpha_j\}_{j=1,\ldots, 2\mathsf g}$.  
Choosing $k_1$ sufficiently large, we may assume that 
$\Gamma$ is torsion-free. Choose a basis $\gamma_1,\ldots, \gamma_r$ of $\Gamma$, 
and write 
$$
\alpha_j=\prod_{i=1}^r \gamma_i^{a_{ij}},
$$ 
with  $a_{ij}\in \Z$. 
Recall that all $\alpha_j$ are Weil numbers, 
i.e., all Galois-conjugates of $\alpha_j$ have absolute value $\sqrt{q}$, where $q=\# k_1$.  
It follows that, for $j\neq j'$, either $\alpha_j=\alpha_{j'}$ or 
$\alpha_j, \alpha_{j'}$ generate a subgroup of rank two in $\Gamma$ 
(since $\Gamma$ does not contain torsion elements). 
We get a subdivision of the sequence of eigenvalues 
$$
\{\alpha_j\}_{j=1,\ldots, 2\mathsf g} =\sqcup_{s=1}^t I_s,\quad t\le 2\mathsf g, 
$$
into subsets of equal elements. 
Put $d_s = \# I_s$ and let $\alpha_s\in I_s$.

\begin{theorem} 
\label{thm:isoge}
Let $A$ and $\tilde{A}$ be abelian varieties of dimension $\mathsf g$
over finite fields $k_1$, resp. $\tilde{k}_1$. 
Let $F$, resp. $\tilde{F}$, be a simple linear recurrence as in equation~\eqref{eqn:fn}.
Assume that $F(n)\mid \tilde{F}(n)$ for infinitely many $n\in \N$.
Then $A$ and $\tilde{A}$ are isogenous. 
\end{theorem}

\begin{proof}
Let $\Gamma\in \C^*$ be the (multiplicative) subgroup generated by 
$\{\alpha_j\}\cup\{\tilde{\alpha}_j\}$. Enlarging $k_1$, resp. $\tilde{k}_1$, 
we may assume that $\Gamma$ is torsion-free. 
Proposition~\ref{prop:cz2} implies that $F/\tilde{F}$ is a simple linear recurrence.

The Laurent polynomial corresponding to $F$, resp. $\tilde{F}$,  has the form
$$
\prod_{s=1}^{t}(\prod_{i=1}^r x_{i}^{a_{is}} -1)^{d_s}, \quad \text{ resp. }\,\,\,
\prod_{\tilde{s}=1}^{\tilde{t}}(\prod_{i=1}^r x_{i}^{\tilde{a}_{i{\tilde{s}}}} -1)^{d_{\tilde{s}}}.
$$
Observe, that 
$$
\gcd(\prod_{i=1}^r x_{i}^{a_{is}} -1, \prod_{i=1}^r x_{i}^{a_{is'}} -1)\in \C^*,
$$ 
for $s\neq s'$. The same holds for $\tilde{F}$. 
Using Lemma~\ref{lemm:poly}, we conclude that 
$t=\tilde{t}$, that we can order the indices so 
that $\# I_s=\# \tilde{I}_s$,
and so that the multiplicative groups generated by $\alpha_s\in I_s$ 
and $\tilde{\alpha}_s\in \tilde{I}_s$ have rank 1, for each $s=1,\ldots, t$.
Thus $\tilde{\alpha}_s=\alpha_s^{u}$, where $u\in \Q$ depends only on $k_1$ and $\tilde{k}_1$. 
It follows that some integer powers of $\Fr, \tilde{\Fr}$ have the same sets of eigenvalues, with 
equal multiplicities. It suffices to apply Theorem~\ref{thm:tate} to conclude
that $A$ is isogenous to $\tilde{A}$.  
\end{proof}

\section{Reconstruction}
\label{sect:isog}

We return to the setup in Section~\ref{sect:intro}:
$C,\tilde{C}$ are irreducible smooth projective curves over 
$k$ of genus $\ge 2$, with Jacobians
$J$, resp. $\tilde{J}$. We have a diagram

\

\centerline{
\xymatrix{
J(k)\ar[d]_{\phi^0}    & J^1(k)\ar[d]_{\phi^1}     & \ar[l]_{j_1}\ar[d]_{\phi_s} C(k) \\
\tilde{J}(k)           & \tilde{J}^1(k) & \ar[l]_{\tilde{j}_1}\tilde{C}(k)
}
}

\

where 
\begin{itemize}
\item $\phi^0$ is an isomorphism of abstract abelian groups;
\item $\phi^1$ is an isomorphism of homogeneous spaces, 
compatible with $\phi^0$;
\item the restriction $\phi_s\,:\, C(k)\ra \tilde{C}(k)$ of $\phi^1$  
is a bijection of sets.
\end{itemize}
It will be convenient to choose 
a point $c_0\in C(k_0)$ and fix the embeddings 
$$
\begin{array}{ccc}
C(k) & \ra & J(k)\\
c    & \mapsto & c-c_0\end{array} \hskip 1cm 
\begin{array}{ccc}
\tilde{C}(k) & \ra & \tilde{J}(k)\\
\tilde{c}    &  \mapsto & \tilde{c}-\phi_s(c_0).
\end{array}
$$
With this choice, the isomorphism of abelian groups $\phi$
induces a bijection on the sets $C(k)$ and $\tilde{C}(k)$. 
In this situation we will say that
$$
\phi\,:\, (C,J)\ra (\tilde{C},\tilde{J})
$$
is an isomorphism of pairs.

\begin{lemma}
\label{lemm:h0}
For any choice of $n_1,\ldots, n_r\in \N$
and $c_1,\ldots, c_r\in C(k)$ 
one has
$$
\dim {\rm H}^0(C, \mathcal O(\sum_i n_i c_i))=
\dim  {\rm H}^0(\tilde{C}, \mathcal O(\sum_i n_i \phi^0(c_i)).
$$
\end{lemma}

\begin{proof}
The effectivity of a divisor on $C$ is 
intrinsically determined by the group $J(k)$: the images of the maps
$C^{(d)}\ra J$, resp. $\tilde{C}^{(d)}\ra \tilde{J}$, 
are the same (under $\phi^0$). We can distinguish $D\in J(k)$
with $\dim {\rm H}^0(C, D)\ge 1$, and therefore all sets of linearly
equivalent divisors. By induction, we can detect that $\dim {\rm H}^0(C,D)\ge n$, 
with $n>1$: there are infinitely many points 
$c\in C(k)\subset J(k)$ such that $\dim {\rm H}^0(C,D-c)\ge n-1$.
 \end{proof}

\begin{corollary}
\label{coro:hyper}
If $C$ is hyperelliptic, trigonal or special (i.e., violate
the Brill--Noether inequality) than so is $\tilde{C}$.  
\end{corollary}

\begin{corollary} 
Let $A\subset C^{(d)} \hookrightarrow J$, for  $d<\mathsf g$, 
be a proper abelian subvariety. 
Then there is a proper abelian subvariety 
$\tilde{A}\subset \tilde{C}^{(d)}\hookrightarrow \tilde{J}$
such that $\phi^0$ induces an isomorphism between $A$ and $\tilde{A}$.
\end{corollary}

\begin{proof}
Any such abelian subvariety of maximal dimension is characterized by the
property that it contains an arbitrarily large abelian subgroup of rank equal to twice its
dimension. In particular, $\phi^0$ induces an isomorphism on such subvarieties.
\end{proof}

\begin{lemma} 
\label{lemm:bielli}
Assume that $\mathsf g(C) > 2$ and that $C$ is bielliptic. Then
$\tilde{C}$ is also bielliptic and the map $\phi^0$ 
commutes with every bielliptic involution on $C$ and $\tilde{C}$, 
respectively.
\end{lemma}

Recall that a bielliptic structure is a map $j_E: C\to E$ of degree $2$,
where $E$ is an elliptic curve.
By Theorem~\ref{thm:ah}, all bielliptic structures correspond
to embedded elliptic curves $E\subset C^{(2)}\subset J$.
Since we assume $\mathsf g(C) > 2$, there is a finite number of such embeddings
and they are preserved under $\phi^0$. Thus if $C$ is bielliptic
then so is $\tilde{C}$,  and the groups generated by bielliptic
reflections are isomorphic.

\begin{corollary} 
\label{coro:klein}
If $C$ is the Klein curve then $\tilde{C}$ is also a Klein curve.
Indeed, this is a unique curve of genus $3$ which has the action
of $\PGL_2(\mathbb F_7)$. 
The action is generated by bielliptic involutions
and hence $\tilde{C}$ is isomorphic to $C$.
\end{corollary}

\begin{remark}
Note that the isomorphism $\phi^0$ itself does not have to be
algebraic, a profinite power of the Frobenius will have the same properties.
\end{remark}

Assume that ${\rm char}(k_0)\neq 2$, and that $\# k_0$ is 
sufficiently large, i.e., for all finite extensions $k_1/k_0$
the points $C(k_1)$ generate $J(k_1)$, and same for $\tilde{C}$.

\begin{lemma}
\label{lemm:frobe}
Assume that $C$ and $\tilde{C}$ are not hyperelliptic. 
Fix finite fields $k_0,\tilde{k}_0$ such that 
$\#k_0,\#\tilde{k}_0$ are sufficiently large and 
$J(k_0)\subset \tilde{J}(\tilde{k}_0)$. 
Consider the tower of field extensions:
$k_0\subset k_1\subset \ldots $, where $k_i/k_{i-1}$ is 
the unique extension of degree 2, and similarly for $\tilde{k}_0$.  
Then, for all $n\in \mathbb N$, 
$$
\phi^0(J(k_n))\subset \tilde{J}(\tilde{k}_n).
$$
\end{lemma}

\begin{proof}
We have an intrinsic inductive characterization of 
$C(k_n)$ and $J(k_n)$, resp. $\tilde{C}(\tilde{k}_n)$ and 
$\tilde{J}(\tilde{k}_n)$. 
Namely, $c\in C(k_n)\setminus C(k_{n-1})$, 
iff there exists a point $c'\in C(k)$ such that $c+c'\in J(k_{n-1})$. 
Indeed, if $c\in C(k_n)$ then $c'$ is the conjugate for the Galois
automorphism $\sigma$ of $k_n/k_{n-1}$. Conversely, if
$c+c'$ is a pair as above and $\sigma(c)\neq c'$, then 
$\sigma(c+c')=c+c'\in J(k)$, which defines a  nontrivial 
hyperelliptic pencil on $C$, contradicting our assumption.
By assumption on $k_0$, points $C(k_n)$ generate $J(k_n)$, as an abelian group. 
By induction, it follows that $\phi^0(J(k_n))\subset \tilde{J}(\tilde{k}_n)$.
\end{proof}

By Corollary~\ref{coro:hyper}, the hyperelliptic property of $C$ implies the same 
for $\tilde{C}$. The hyperelliptic 
case requires a more delicate analysis of point configurations.

Let $C$ be a hyperelliptic curve over a finite field $\mathbb F_q$. 
The Jacobian $J^2$
of zero cycles of degree 2 contains a unique effective zero-cycle 
$z_0\in J^2(\mathbb F_q)$ corresponding to the hyperelliptic pencil on $C$.  
We use this cycle to identify $J^2(k)\simeq J(k)=J^0(k)$. 
Let $k_0/\mathbb F_q$ be a finite extension, 
$k_1/k_0$ a quadratic extension and $\sigma$ the nontrivial element of the Galois group
$\Gal(k_1/k_0)$. Put
$$
C(k_1)^{-}:=\{ c\in C(k_1)\,|\, \sigma(c) + c = z_0 \in J^2(k_0)\}.
$$

\begin{lemma}
\label{lemm:minus}
Let $C$ be a hyperelliptic curve defined over $\mathbb F_q$. 
Then there exists an $N\in \N$ such that for all finite extensions 
$k_0/\mathbb F_q$ with $q^N\mid \# k_0$, 
the zero-cycles of even degree with support in 
$C(k_1) \setminus  C(k_1)^{-}$ generate $J(k_1)\simeq J^2(k_1)$.  
\end{lemma}

\begin{proof}
Let $H\subset J(k_1)$ be the subgroup generated by zero-cycles of even degree with support in 
$C(k_1) \setminus  C(k_1)^{-}$. 
Put $q:=\#k_0$. Note that 
$$
|\# C(k_1)^{-} - q |\le 2\mathsf g \sqrt{q}.
$$
Indeed, let $\iota \,:\, C\ra \P^1$ be the hyperelliptic projection. 
Then $\iota(C(k_1)^{-})\subseteq \P^1(k_0)$, and the image corresponds to 
those points on  $b\in \P^1(k_0)$ such that the degree 2 cycle $\iota^{-1}(b)$
does not split over $k_0$. The claim follows from standard Weil estimates. 
Lemma~\ref{lemm:pre-minus} implies a universal ($k_1$ independent) 
bound for the index $I:=[J(k_1):H]$, e.g., $I < m$.

Now we apply the argument of Lemma~\ref{lemm:ell-points}.
Let $k_0$ be such that $J(k_0)$ contains
all $J(k)[\ell]$, for $\ell < m$. Then $H= J(k_1)$.
Indeed, for $\ell\neq 2$ and $J(k)[\ell]\subset J(k_0)$ 
the order of $J(k_1)/J(k_0)$ is coprime to $\ell$:
if an automorphism of order $2$ acts
trivially on $J(k)[\ell]$ then it also acts trivially on all 
elements of $\ell$-power order in $J(k_1)$. 
Next, note that the elements of the 
form $\frac{1}{2}x, x\in J(k_0)$ generate
the $2$-primary part of $J(k_1)$ but that $\sigma(\frac{1}{2}x) =x + z_0, z_0\in J^2(k_0)$
and hence $\frac{1}{2}x$ is never in 
$J(k_1)^{-}$ (the subgroup generated by  $C(k_1)^{-}$).
This completes the argument for $\ell=2$. 
\end{proof}

\begin{lemma}
\label{lemm:hyper-ell-frob}
Assume that $C$ and $\tilde{C}$ are hyperelliptic. 
There exist finite fields $k_0,\tilde{k}_0$ 
and towers of  quadratic field extensions:
$k_0\subset k_1\subset \ldots $, resp.  for $\tilde{k}_0$, such that  
for all $n\in \mathbb N$  
$$
\phi^0(J(k_n))\subset \tilde{J}(\tilde{k}_n).
$$
\end{lemma}

\begin{proof}
By Lemma~\ref{lemm:minus}, the points in $C(k_i)\setminus C(k_i)^{-}$
generate $J(k_i)$. This subset of points is defined intrinsically in $C(k)$, 
provided $J(k_{i-1})$ is already known. By induction, as in the proof of Lemma~\ref{lemm:frobe}, 
we obtain the required tower of degree 2 extensions, with an embedding 
$$
\phi^0\,:\, J(k_i) \ra \tilde{J}(\tilde{k}_i).
$$  
\end{proof}

\begin{theorem}
\label{thm:main3}
Let $\phi\,:\, (C,J)\ra (\tilde{C},\tilde{J})$ be an isomorphism of pairs. 
Then $J$ and $\tilde{J}$ are isogenous. 
\end{theorem}

\begin{proof}
In both hyperelliptic and nonhyperelliptic case we have shown that, for
sufficiently large finite ground fields $k_0,\tilde{k}_0$, there exist
towers $\{ k_n\}_{n\in \N}$ and $\{ \tilde{k}_n\}_{n\in \N}$ 
of degree 2 field extensions with the following property:
$$
\phi^0(J(k_n))\subset \tilde{J}(\tilde{k}_n)
$$
(see Lemma~\ref{lemm:frobe} and Lemma~\ref{lemm:hyper-ell-frob}).
Now we apply Theorem~\ref{thm:isoge} to 
the Frobenius automorphisms $\Fr,\tilde{\Fr}$.   
\end{proof}

\section{Generalized Jacobians}
\label{sect:background}

Let $k$ be a field and $K/k$ a field extension. 
Then the set 
$$
\P(K):=K^*/k^* = (K\setminus 0)/k^*
$$ 
carries the structure of
an abelian group {\em and} a projective space. Moreover, the projective structure
(the set of projective lines, their intersections etc.) is compatible with 
the group operation. 
We have:

\begin{theorem}
\label{thm:proj}
Let $K$ and $\tilde{K}$ be function fields over $k$. 
Assume that we have an isomorphism of abelian groups 
$$
\phi_{gr}\,:\, K^*/k^*\ra \tilde{K}^*/k^*
$$ 
inducing an isomorphism of projective structures
$$
\phi_{pr}\,:\,  \P(K)\ra \P(\tilde{K}).
$$
Then there exists an isomorphism of fields
$$
\phi_f\,:\, K\ra \tilde{K}.
$$
\end{theorem}

\begin{proof}
See \cite{bt}, Section 4, for precise definitions and a proof. 
 \end{proof}

We will apply this Theorem to $K=k(C)$ and $\tilde{K}=k(\tilde{C})$. 

\

We write $J=J$ and $J^1=J^1$, resp. $\tilde{J}$ and $\tilde{J}^1$, 
for the Jacobian of degree 0 and degree 1 zero-cycles on $C$, resp. $\tilde{C}$.
Let $Z^0(k)=Z^0(C(k))$, resp. $\tilde{Z}^0(k)$, be the group of
degree 0 zero-cycles on $C$, resp. $\tilde{C}$. 
We have a diagram 

\

\centerline{
\xymatrix{
1\ar[r] & K^*/k^*       \ar[r] & Z^0(k)\ar[d]^{\phi^0}\ar[r] &\ar[d]^{\phi} J(k)\ar[r] & 1\\
1\ar[r] & \tilde{K}/k^* \ar[r] & \tilde{Z}^0(k) \ar[r]       &\tilde{J}^0(k) \ar[r] & 1
}
}

\

\noindent
where $\phi_0$ is an isomorphism of abelian groups induced by $\phi$ and $\phi_s$. 
This implies the following

\begin{lemma}
\label{lemm:ab-gr}
Under the assumptions of 
Conjecture~\ref{conj:main}, we have an isomorphism of abelian groups
$$
\phi_{gr}\,:\, K^*/k^* \ra \tilde{K}^*/k^*.
$$
\end{lemma}

We have a natural embedding $C \hookrightarrow J^1$ from \eqref{eqn:11}, 
mapping a point in $C(k)$ to its cycle-class. 
Choosing a point $c_0\in C(k)$ we also have an embedding 
\begin{equation}
\label{eqn:00}
\begin{array}{ccc}
C(k) & \hookrightarrow & J(k)\\
c& \mapsto & [c-c_0]
\end{array}
\end{equation}

Write 
$\mathsf m=\sum_{i=1}^r n_i P_i$, $n_i\in \Z\setminus 0$, $P_i\in C(k)$ 
for a zero-cycle on $C$ and $|\mathsf m|:=P_1\cup \ldots \cup P_r$ for 
its support. Let $Z^0_{\mathsf m}(k)=Z^0_{\mathsf m}(C(k))$ 
be the group of degree 0 zero-cycles with
support disjoint from the support of $\mathsf m$.

For every effective $k$-rational zero-cycle $\mathsf m$ let 
$J_{\mathsf m}(k)$, resp. $J^1_{\mathsf m}$, 
be the {\em generalized Jacobian} of degree $0$, resp. degree 1,  
zero-cycles on $C$ over $k$, modulo
the ideal generated by $\mathsf m$.  
The generalized Jacobian $J_{\mathsf m}$ is an algebraic group, fibered over the Jacobian $J$
with fibers connected abelian linear algebraic groups of dimension 
$\# \mathsf m -1$.

We have a compatible family of embeddings 
\begin{equation}
\label{eqn:mu1}
\mu^1_{\mathsf m}\,:\, C(k)\setminus |\mathsf m| \hookrightarrow J^1_{\mathsf m}(k)
\end{equation}
as well as surjective homomorphisms of abelian groups
\begin{equation}
\label{eqn:mu0}
\mu_{\mathsf m}\,:\, Z^0_{\mathsf m}(k)\ra J_{\mathsf m}(k).
\end{equation}

The natural embedding $C\hookrightarrow J^1$, assigning to a point $c\in C(k)$ 
its cycle class $[c]\in J(k)$ extends uniquely to a compatible family of maps
$$
\mu_{\mathsf m}\,:\, C\setminus |\mathsf m| \hookrightarrow J^1_{\mathsf m}.
$$

\begin{proposition}
\label{prop:1}
Assume that for all $\mathsf m=P+R$ one has an isomorphism 
of abelian groups
$$
\phi^0_{\mathsf m}\,:\, J_{\mathsf m}(k)\ra \tilde{J}_{\mathsf m}(k)
$$
and a diagram (of compatible maps)

\

\centerline{
\xymatrix{
C(k)\setminus |\mathsf m|\ar[d] \ar[r]^{\,\,\,\,\,\mu_{\mathsf m}} 
& J^1_{\mathsf m}(k) \ar[d]_{\varphi_{\mathsf m}}\ar[r]^{\phi^1_{\mathsf m}} & 
     \tilde{J}^1_{\mathsf m}(k) \ar[d]^{\tilde{\varphi}_{\mathsf m}}  & 
\ar[d]\ar[l]_{\!\tilde{\mu}_{\mathsf m}}\tilde{C}(k)\setminus |\mathsf m|\\
C(k)  \ar[r]                      & J^1(k) \ar[r]^{\phi^1}              & 
                        \tilde{J}^1(k)                                        & \tilde{C}(k)\ar[l] 
}
}

\

\noindent
with $\phi^1, \phi^1_{\mathsf m}$ isomorphisms 
of homogeneous spaces under 
$J_{\mathsf m}(k)\simeq \tilde{J}_{\mathsf m}(k)$ 
inducing bijections of sets 
$$
C(k)\setminus \mathsf m =  \tilde{C}(k)\setminus \mathsf m
$$
Then $k(C)=k(\tilde{C})$.
\end{proposition}

\begin{proof}
It suffices to prove that the isomorphism of groups 
$$
\phi_{gr}\,:\, K^*/k^* \ra \tilde{K}^*/k^*
$$
established in Lemma~\ref{lemm:ab-gr} preserves the respective
projective structures. 
A projective line in $\P(K)$ is the projectivization of a two-dimensional
$k$-vector space generated by two functions 
$f,g\in K^*$. By the compatibility with multiplication in $K^*/k^*$,
we can assume that $g=1$. 

Every pair of rationally equivalent effective zero-cycles 
$z_0,z_{\infty}$ on $C(k)$ defines 
a unique ``point'' in $K^*/k^*=\P(K)$ - 
the divisor of a function $f$ with zeroes $z_0$ and poles $z_{\infty}$. 
To get a projective line $\P^1\subset \P(K)$
consider the map $\pi_f\,:\, C\ra \P^1$ defined by $f$ and  
the induced family of cycles $z_{\la}:=\pi_f^{-1}(\la)$, 
for $\la\in \P^1$. The family of ``points'' in $\P(K)$,
given by cycles $z_{\la}-z_{\infty}$, 
is a projective line in $\P(K)$ through $1$ and $f$.

Conversely, assume that for every such pair 
$z_0,z_{\infty}$ of equivalent effective cycles
on $C$ we have an intrinsic definition of the set $z_{\la}$. 
Then we recover the projective structure on $\P(K)$.    

We define an equivalence relation on $C(k)$:
$$
P\sim Q \Leftrightarrow (z_0-z_{\infty}) \in \Ker(\mu_{\mathsf m}),
$$
where $\mathsf m=P+Q$. Observe that $P\sim Q$ iff $P,Q$ are both 
contained in the fiber of $\pi_f$.
The set $\{z_{\la}\}$ 
is intrinsically defined as the set of equivalence classes 
with respect to ``$\sim$''. 
 \end{proof}

The group $Z^0(k)$ is the set of $k$-points of a countable union of 
algebraic varieties
$$
Z^0(k)=\cup_{n\in \N} (C^{(n)}\times C^{(n)})^{\circ}(k),
$$
each parametrizing pairs of effective zero-cycles of degree $n$ of disjoint support. 
This induces a grading on the subset $K^*/k^*\subset Z^0(k)$. 

\begin{lemma}
\label{lemm:grad}
The multiplicative group $K^*/k^*$ is generated by components of degree $\le \mathsf g+1$. 
\end{lemma}

\begin{proof}
Let $z,z'$ be effective cycles on $C$, of disjoint support, 
of degree $n> \mathsf g+1$. 
The space of cycles equivalent to $z-z'$ has (projective) dimension 
$\ge 2$. Then there exists an effective cycle $z''$, equivalent to $z$ and $z'$,
such that $|z|\cap |z'|\neq \emptyset $ and $|z'|\cap| z''|\neq \emptyset$. 
(Choose a hyperplane section in $\P({\rm H}^0(C,\mathcal O(z-z')))$ which 
contains points from the support of both $z$ and $z'$. The intersection
of $C$ with this hyperplane gives $z''$.)  
Now we can write
$$
z-z'=(z-z'') +(z'-z'')
$$
with $(z-z')$ and $(z-z'')$ having degree $<n$. 
 \end{proof}

Let $Y^{\mathsf g+1}$, resp.  $Y^{\mathsf g+1}_{\mathsf m}$,
be the subvariety of
$C^{(\mathsf g+1)}\times C^{(\mathsf g+1)}$ corresponding to pairs of
effective degree $\mathsf g+1$ cycles of disjoint support, resp. in addition
with support disjoint from $|\mathsf m|$.
Each such pair of cycles $(z_0,z_{\infty})$ determines a principal divisor on $C$,
thus a function $f$, modulo $k^*$,
and an algebraic morphism $Y^{\mathsf g+1}_{\mathsf m}\ra \gm$ 
which on the level of points is given by
$$
\begin{array}{ccc}
Y^{\mathsf g+1}_{\mathsf m}(k) & \ra & k^* \\
(z_0,z_{\infty})& \mapsto & f(P)/f(Q).
\end{array}
$$
The algebraic variety $X$, resp. $X_{\mathsf m}$ is defined as the preimage of 
$1\subset \gm$.
We get the diagram, with maps morphisms of algebraic varieties:

\centerline{
\xymatrix{
1\ar[r] & X_{\mathsf m} \ar[d]\ar[r] & Y^{\mathsf g+1}_{\mathsf m}\ar[r]\ar[d] & \ar[r]\ar[d] J_{\mathsf m} & 1\\
1\ar[r] & X \ar[r]   &  Y^{\mathsf g+1} \ar[r]& \ar[r] J & 1\\
}
}

\

\noindent
and a similar diagram for $\tilde{C}$.

\begin{lemma}
Under the assumptions of Conjecture~\ref{conj:main}, we 
have a bijection of sets:
\begin{itemize}
\item $\phi^X_{s}\,:\, X(k) \ra \tilde{X}(k) $
\item $\phi^X_{s,\mathsf m}\,:\, X_{\mathsf m}(k) \ra \tilde{X}_{\mathsf m}(k)$. 
\end{itemize}
\end{lemma}

\begin{proof}
It suffices to give the following intrinsic description:
$$
Y^{\mathsf g+1}_{\mathsf m}(k):=\{ (z_{0},z_{\infty})\,|\, P,Q\notin |z_{0}|\cup |z_{\infty}|
\,\, \text{ and } \,\, P+Q\in |z_{\la}|, \text{ for some } \la \} 
$$
and a similarly for $Y^{\mathsf g+1}(k)$.
 \end{proof}

\section{Anabelian geometry}
\label{sect:anabel}

In this section we discuss an application of the above results to 
Grothendieck's Anabelian Program - the reconstruction of function fields from 
Galois groups.

\

Let $C$ be an irreducible smooth projective curve 
over $k=\bar{\mathbb F}_p$ of genus $\mathsf g\ge 2$, $J$ its Jacobian 
and $K=k(C)$ its function field. 
Throughout, we assume that $p>2$. 
Fix an algebraic closure $\bar{K}/K$ and let $\G=\G_K=\Gal(\bar{K}/K)$ be 
the absolute Galois group.
The main idea of anabelian geometry is that
$\G$, or even one of its factors, determines $C$. 
Note that $\G$ is the completion of a 
free group with an infinite number of generators. 
In particular, for any two curves over $k$ 
the corresponding groups are isomorphic as abstract topological groups.
However, we will see that in some instances 
additional structures allow us to 
recover the curve from the Galois group.

Let 
$$
\G^a=\G/[\G,\G]
$$ 
be the abelianization of $\G$.
Let $\ell$ be a prime number, $\G_{\ell}$ the $\ell$-completion of $\G$, 
and $\G^a_{\ell}$ the image of $\G_{\ell}$ in the abelianization.  
Clearly, $\G^a=\prod_{\ell} \G^a_{\ell}$. 
A $k$-rational point $c\in C(k)$ 
determines a discrete rank one valuation $\nu=\nu_c$ of the function
field $K$. 
We write $\I_{\nu}\subset \G$ for the corresponding inertia
subgroup and $\I^a_{\nu}$, resp.  $\I^a_{\nu,\ell}$,
for its image in $\G^a$, resp. $\G^a_{\ell}$.
The group $\I^a_{\nu}$ is topologically cyclic.

We now proceed to describe the 
groups $\G^a_{\ell}$, for $\ell\neq p$
(the structure of $\G^a_p$ is more refined), 
closely following Sections~9 and 11 of \cite{bt}.  
Dualizing the exact sequence
$$
0\ra K^*/k^*\ra \Div(C)\ra \Pic(C)\ra 0
$$
we obtain the sequence
\begin{equation}
\label{eqn:crrr}
0\ra \Z_{\ell} \stackrel{\Delta_{\ell}}{\lra} \cM(C(k),\Z_{\ell})\ra 
\G^a_{\ell}\ra \Z_{\ell}^{2\mathsf g}\ra 0,
\end{equation}
with the identifications 
\begin{itemize}
\item $\Hom(\Pic(C),\Z_{\ell})=\Delta_{\ell}(\Z_{\ell})$ 
(since $J(k)=\Pic^0(C)$ is torsion);
\item $\cM(C(k),\Z_{\ell})=\Hom(\Div(C),\Z_{\ell})$ 
is the $\Z_{\ell}$-linear space of maps from $C(k)\ra \Z_{\ell}$
(note that $\Div(C)$ can be viewed as
the free abelian group generated by points in $C(k)$);
\item $\Z_{\ell}^{2\mathsf g}=\Ext^1(J(k),\Z_{\ell})$.
\end{itemize}
The interpretation
\begin{equation}
\label{eqn:kummer}
\G^a_{\ell}=\Hom(K^*/k^*,\Z_{\ell}),
\end{equation}
arising from Kummer theory allows us to identify 
\begin{equation}
\label{eqn:cm}
\G^a_{\ell}\subset \cM(C(k),\Q_{\ell})/\text{constant maps} 
\end{equation}
as the $\Z_{\ell}$-linear subspace of 
maps $\mu\,:\, C(k)\ra \Q_{\ell}$ (modulo constant maps)
such that
$$
[\mu,f]\in \Z_{\ell} \,\, \text{ for all }\,\,  f\in K^*/k^*.
$$
Here $[\cdot ,\cdot ]$ is the pairing: 
\begin{equation}
\label{eqn:inter}
\begin{array}{rcl}
\cM(C(k),\Q_{\ell})\times K^*/k^* & \ra     &       \Q_{\ell} \\ 
       (\mu , f )                 & \mapsto &  [\mu,f]:=\sum_{c} \mu(q)f_c,
\end{array}   
\end{equation}
where $\dv(f)=\sum_c f_c c$.
In this language, an element of an 
inertia subgroup $\I^a_{\nu,\ell}\subset \G^a_{\ell}$
corresponds to a ``delta''-map (constant outside the point $c=c_\nu$). 
Each $\I^a_{\nu,\ell}$ has a canonical (topological) generator $\delta_{\nu,\ell}$
and the (diagonal) map 
$\Delta\in \cM(C(k),\Q_{\ell})$ from \eqref{eqn:crrr}
is given by 
$$
\Delta_{\ell} =\sum_{c\in C(k)} \delta_{c,\ell}.
$$

\

Consider the abelian Galois group $\G^a_{(p)} = \prod_{\ell\neq p} \G_\ell^a$.
Let $\I=\{ \I^a_c\} $ be the set of 1-dimensional
valuation subgroups $\I^a_c\subset \G^a$ corresponding to points $c\in C(k)$.

\begin{conjecture}
\label{conj:galois}
Let $C$ be a curve of genus $\mathsf g(C) \geq 2$ over $k=\bar{\mathbb F}_p$.
The pair $(\G^a_{(p)},\I)$ determines the function field $k(C)$, modulo isomorphisms.
\end{conjecture}

\begin{remark}
This fails when $\mathsf g(C) = 1$. 
For any two elliptic curves over $k$ the
pairs $(\G^a_{(p)},\I)$ are isomorphic.
There are two types: 
supersingular curves with $J\{p\}=0$ (which are all isogenous) 
and ordinary curves.  
\end{remark}

We have the following partial result:

\begin{theorem}
\label{thm:galois}
Let $C, \tilde{C}$ be curves of genus $\ge 2$ over $k=\bar{\mathbb F}_p$, 
with $p>2$.  
Assume that there is an isomorphism of pairs 
$$
(\G^a, \I) \stackrel{\sim}{\longrightarrow} (\tilde{\G}^a, \tilde{\I}).
$$
Assume in addition that either 
\begin{itemize}
\item $J\{p\} = 0$ or
\item $\mathsf g(C)> 4$.
\end{itemize}
Then there is an isogeny $J\ra \tilde{J}$ and 
$$
\End_k(J)\otimes \Z_\ell = \End_k(\tilde{J})\otimes \Z_\ell,
$$
for all $\ell\neq p$.
\end{theorem}

The condition $\mathsf g(C) > 4$ arises as follows:
in the nonsupersingular case when $J\{p\}\neq 0$,  
our argument will be based on a detailed understanding of the map
$$
\begin{array}{ccc}
j \,:\, C^{(2)}\times C^{(2)} &  \to    &  J\\
           ((x,x'), (y,y')) & \mapsto & (x+x')-(y+y').
\end{array}
$$
In particular, it will be essential to 
describe all abelian varieties in the image of $j$. 
The corresponding classification is carried 
out in Section~\ref{sect:append}, 
under the assumption that $\mathsf g(C)>4$.

\begin{proof}[Proof of Theorem~\ref{thm:galois}]
We will reduce to a version of Theorem~\ref{thm:main},
following closely the description of Galois groups in \cite{bt}, Section 11. 

Dualizing \eqref{eqn:kummer}, we recover the pro-$\ell$-completion $\hat{K}^*_{\ell}$
of the multiplicative group $K^*/k^*$ as $\Hom(\G^a_K, \Z_{\ell})$. 
Consider the following exact sequences  
\begin{equation}
\label{eqn:seqq}
0\ra K^*/k^*\stackrel{\rho_C}{\lra} \Div^0(C)\stackrel{\pic}{\lra} J(k)\ra 0, 
\end{equation}
\begin{equation}
\label{eqn:seqq-times}
0\ra K^*/k^*\otimes \Z_{\ell}\stackrel{\rho_{C,\ell}}{\lra} 
\Div^0(C)\otimes \Z_{\ell}\stackrel{\pic_{\ell}}{\lra} 
J\{\ell\}\ra 0. 
\end{equation}
Put
$$
\cT_{\ell}(C):=\lim_{\longleftarrow}{\rm Tor}_1(\Z/\ell^n,J\{\ell\}).
$$ 
We have $\cT_{\ell}(C)=\Z_{\ell}^{2\mathsf g}$, where $\mathsf g$ is the genus of $C$.
Passing to pro-$\ell$-completions in \eqref{eqn:seqq} we obtain an exact sequence
of torsion-free groups
\begin{equation}
\label{eqn:seqq-pro-ell}
0\ra \cT_{\ell}(C)\ra \hat{K}^*_{\ell}\stackrel{\hat{\rho}_C}{\lra} \widehat{\Div^0}(C) \lra 0, 
\end{equation}
since $J(k)$ is an $\ell$-divisible group. 
We write 
$\widehat{\Div^0}(C)_{\ell}$ for the $\ell$-completion of
$\Div^0(C)$. 
Clearly, $\Div^0(C)\otimes \Z_{\ell}\subset \widehat{\Div^0}(C)_{\ell}$ and we have a diagram   
\begin{equation}
\label{eqn:dia-need}
\begin{array}{ccccccccccc}
  &     &    0         & \ra & K^*/k^*\otimes\Z_{\ell} &  \stackrel{\rho_{C,\ell}}{\lra}   
&   \Div^0(C)\otimes \Z_{\ell}  &\stackrel{\pic_{\ell}}{\lra}  & J\{\ell\}  & \ra &  0\\ 
  &     &              &     &   \downarrow             &                 &   \downarrow      &     &      \downarrow &     &   \\
0 & \ra &  \cT_{\ell}(C)    & \ra & \hat{K}^*_{\ell}& \stackrel{\hat{\rho}_{C,\ell}}{\lra} & 
\widehat{\Div^0}(C)_{\ell}& 
\stackrel{\hat{\pic}_{\ell}}{\lra} & 0.  & 
\end{array}
\end{equation}
Fix a valuation $\nu_0$ and a generator $\delta_{\nu_0}$ of $\I^a_{\nu_0}$. 
This gives a canonical identification of 
generators $\delta_{\nu}\in \I^a_{\nu}$ for all $\nu\neq \nu_0$, and similarly 
all $\delta_{\nu,\ell}\in \I^a_{\nu,\ell}$, for all $\ell\neq p$.

Let $\FS(C)_{\ell}\subset\hat{K}^*_{\ell} = \Hom(\G^a_{\ell}, \Z_{\ell}) $ 
be the subgroup topologically generated by elements 
$\alpha_{\ell} \in \hat{K}^*_{\ell}$ such that 
\begin{itemize}
\item $\alpha(\delta_{\nu_0,\ell})=1$, 
\item $\alpha(\delta_{\nu,\ell})=-1$, for some $\nu\neq \nu_0$. 
\item $\alpha_{\ell}$ is trivial on all $\delta_{\nu',\ell}$, for $\nu'\neq \nu, \nu_0$.  
\end{itemize}
The group $\FS(C)_{\ell}$ is equal to $\Div^0(C)\otimes \Z_{\ell}$ and we get a sequence
\begin{equation}
\label{eqn:fs}
0\to \cT_{\ell}(C)\to \FS(C)_{\ell} \to J\{\ell\}\to 0,
\end{equation}
with cohomology isomorphic to $K^*/k^*\otimes \Z_{\ell}$.
Sequence~\eqref{eqn:fs} defines a map
$$
C(k)\stackrel{\iota_{\ell}}{\lra} J\{\ell\}.
$$
Combining these, we obtain a map 
$$
\iota =\prod_{\ell\neq p} \iota_{\ell} \,:\, C(k)\ra  J(k)/J\{p\}.
$$
If $J\{p\}=0$, the map $\iota$ is an embedding and we
can apply Theorem~\ref{thm:main} to conclude that
the Galois isomorphism implies isogeny. 

\

Assume that $J\{p\}\neq 0$ and $\mathsf g(C)>4$.

\begin{lemma}
\label{lemm:j-p}
Let $C$ be a curve of genus $\mathsf g(C)>4$. 
If $j(C^{(2)}\times C^{(2)})\cap J\{p\}$ 
is infinite then $j(C^{(2)}\times C^{(2)})\subset J$ 
contains a nontrivial abelian variety.
\end{lemma}

\begin{proof}
Consider the action of $\Z_p$ generated by the $p$-component of some
power of the Frobenius endomorphism.
Under the assumptions, the $\Z_p$-orbits of points in $j(C^{(2)}\times C^{(2)})$
can be arbitrarily large. Hence $j(C^{(2)}\times C^{(2)})$ contains arbitrarily
large subsets which are invariant under translations by big cyclic $p$-groups.
Such subsets must be contained in translates of abelian subvarieties in $J$ 
(see Remarks~\ref{rem:genera}, \ref{rem:boxall} and \cite{boxall}). 
\end{proof}

\

\noindent
{\em Case 1: $j(C^{(2)}\times C^{(2)})$ does not contain a 
nontrivial abelian variety.}

\

By Lemma~\ref{lemm:j-p},
$j(C^{(2)}\times C^{(2)})(k)\cap J\{p\}$ is finite.
Let $k_0$ be a sufficiently large finite extension of the ground field 
such that $J(k_0)$ 
contains all points $x,x',y,y' $ with
$(x+x')- (y+y')\in J\{p\}-0$, and a finite subgroup
$A_p\subset J\{p\}$ with trivial action of $\Z_p$ 
(generated by the Frobenius $\Fr$ over $k_0$). We will also assume that 
$A_p\supset J[p]$. 
Put $J_0:= J({k_0})$ and define 
inductively $J_{n+1}$ as the subset of points 
$x\notin J_n\oplus J\{p\}$ such that there is an $x'$ with
$x+x'\in J_n\oplus J\{p\}$. By induction, we obtain
a subgroup $J_{\infty}\subset J(k)$.
Note that $J_{\infty}$ contains $J(k_{\infty})$,  
where $k_{\infty}$ is the 2-closure of $k_0$. 

We claim that $J(k_{\infty})$ is closed
under the above operation. Assume otherwise.  
The action of $\Z_p$ on $J(k_0)$ and hence on $J(k_{\infty})$ is trivial.
Let $x+x'+y_p \in J(k_{\infty})$, with $y_p\in J\{p\}\setminus A_p$.  
Then $\Z_p$ acts nontrivially on $y_p$ but 
trivially on $x+x'+y_p$. Thus $\gamma(x+x')\neq x+x'$, for some $\gamma\in \Z_p$,
and $\gamma(x+x') - (x+x')\in J\{p\}$. 
However, by assumption all such points are contained in 
$J(k_0)$, where the action of $\gamma\in \Z_p$ is trivial, contradiction.

In particular, if $\tilde{C}$ is another curve with the same
data and $\tilde{J}(\tilde k_0)$ contains $J({k_0})$ then
we obtain a tower of inclusions of groups as in Lemma~\ref{lemm:hyper-ell-frob} 
and we can apply Theorem~\ref{thm:main3}.

Let $k_0$ be as above and assume that $0\neq x + x_1\in J(k_0)/J\{p\}$ but that
$x,x_1\notin J(k_1)$. Then $\gamma(x) + \gamma(x_1) = x + x_1\neq 0$, 
which implies that $C$ is hyperelliptic and that $x, x_1$ are
conjugated by a hyperelliptic involution, hence 
$0=x + x_1$, contradicting the assumption on $x,x_1$.
Thus we can proceed as before.

\

\noindent
{\em Case 2: $j(C^{(2)}\times  C^{(2)})$ contains a nontrivial abelian variety.}

\

Assume now that $j(C^{(2)} \times C^{(2)})$ does contain an
abelian variety. In Section~\ref{sect:append} we give
a classification of such subvarieties, when $\mathsf g(C)>4$. 
By Corollary~\ref{coro:D} there is a divisor
$D\subset C^{(2)}$ 
such that any abelian variety in 
$j(C^{(2)}\times C^{(2)})$ is contained
in the image of $D\times C^{(2)} \cap C^{(2)}\times D$.
Thus there exist only finitely many pairs
$x,x'\in C$ and $y,y'\in C$ with $x+x'\notin D$ and $y+y'\notin D$ and
such that $(x+x')- (y+y')\in J\{p\}$.
Let $k_0$ be a sufficiently large finite field so that 
$J(k_0)$ contains all such pairs $x,x'$ and $y,y'$ and so that
$x,\gamma(x)\in C(k_1)$ with $x+ \gamma(x)\notin D$ generate
$J(k_1)$, for any finite extension $k'/k_0$. Then we can apply the same argument as above.

Note that such a field $k_0$ exists, since the number of elements in $D(k')$ grows
as the square root of the number of elements in $C(k')$. 
Hence we can combine the arguments 
Lemma~\ref{lemm:ell-points}, \ref{lemm:pre-minus} and \ref{lemm:minus}
to show that for sufficiently large $k_0$ our assumption holds. 
This finishes the proof of Theorem~\ref{thm:galois}.
\end{proof}

\begin{remark}
When $C$ has an algebraic automorphism $\alpha$ such that $C'=C/\alpha$
has genus $1\le \mathsf g(C')\le \mathsf g(C)/2$ 
we recover the algebraic projection $C\ra J'\subset J$.
In particular, we can recover every bielliptic involution. 

For example, the Klein quartic curve is the unique curve of genus $3$ 
with the maximal number of bielliptic involutions. 
Thus the algebraic structure of the Klein curve
is completely encoded in the pair $(\G^a, \I)$.
Same holds for many other
curves with sufficiently many maps onto curves of small genus,
providing nontrivial examples where Conjecture~\ref{conj:galois}
holds.
\end{remark}

\section{Appendix: Geometric background}
\label{sect:append}

In this section we work over an algebraically closed field $k$ of
characteristic $\neq 2$. 

\

Let $C$ be a curve
of genus $\mathsf g(C)\ge 2$ and $J$ its Jacobian. We will identify the Jacobians
of degree $n$ zero-cycles $J^n$ with $J$. 
Recall that  a $d$-gonal structure 
on $C$ is a surjective morphism $C\ra \P^1$ of degree $d$. 
A hyperelliptic structure on $C$ is a surjective morphism $C\ra \P^1$ of degree 2 and  
a bielliptic structure a surjective morphism $C\ra E$ of degree 2.
If $W_2^1(C)\neq \emptyset$ then $C$ is hyperelliptic and if $W_3^1(C)\neq \emptyset$
and some element from $W_3^1(C)$ defines a proper map then $C$ is trigonal.

\

\
Consider the map
$$
\begin{array}{rcl}
j \,:\, C^{(2)}\times C^{(2)} &  \to    &  J\\
           ((x,x'), (y,y')) & \mapsto & (x+x')-(y+y').
\end{array}
$$
It contracts the diagonal $\Delta=\Delta(C^{(2)})$ 
to a point and the divisor
$$
\delta: =\{ ((x,x'), (y,y'))\in C^{(2)}\times C^{(2)} \,|\, x'=y' \}
$$
to a surface $j(C\times C)$.
Denote the union of the diagonal and the divisor $\delta$ above as $D_{\delta}$.

Assume that
$$
j((x,x'),(y,y'))= j((z,z'),(w,w')),
$$
and that $(x,x')\neq (y,y')$ and  $(x,x')\neq (z,z')$ in $C^{(2)}$.
If  $(x,x',w,w')\neq  (z,z',y,y')$ as elements in $C^{(4)}$
then the relation 
$$
((x+x')-(y+y'))-((z+z')-(w+w')) =0\in J
$$ 
gives a nontrivial
relation of degree $\leq 4$
$$
(x+x'+w +w')- (z+z'+y+y').
$$
Thus if there is a linear series $L$ of degree $\leq 4$ on $C$ with
${\rm H}^0(C,L) \geq 2$, i.e., the variety $W_4^1(C)$ is nonempty.

\begin{lemma}
Assume that
\begin{itemize}
\item $(x,x',w,w') = (z,z',y,y')$ in $C^{(4)}$, 
\item $((x,x'),(y,y'))\neq ((z,z'),(w,w'))$ in $C^{(2)}\times C^{(2)}$, 
\item $(x,x')\neq (y,y')$ and $(z,z')\neq (w,w')$ in $C^{(2)}$.
\end{itemize} 
Then there is a set $\{X,Y,Z,W\}$ which contains both
sets $\{x,x',y,y'\},\{z,z',w,w'\}$ so that after some identification
$$
((x,x'),(y,y')) = ((X,Y),(Y,Z))\,\, \text{ and }\,\, 
((z,z'),(w,w'))= ((X,W),(W,Z)).
$$
in $C^{(2)}\times C^{(2)}$.
Hence
$\{x,x',y,y'\},\{z,z',w,w'\}\in D_{\delta}=\Delta\cup \delta$.
\end{lemma}

\begin{proof}
Since $(x,x',w,w') = (z,z',y,y')$ in $C^{(4)}$ the
four-tuple
$\{x,x',w,w'\}$  contains also all the elements from $\{z,z',y,y'\}$.
Thus if we denote $\{x,x',w,w'\}$ as $(X,Y,Z,W)$ and assume
that none of the pairs $((x,x'),(y,y')),((z,z'),(w,w'))$
consists of the same letters
we obtain that modulo permutation of $(X,Y,Z,W)$
$$
(x,x') =(X,Y), \,\,\, (w,w')=(Z,W), \,\,\, (z,z')= (Y,Z), \,\,\, (y,y') =(W,X),
$$
and that the above
identification is the only possible, modulo permutations.
Thus $\{x,x',y,y'\},\{z,z',w,w'\}\in D_{\delta}=\Delta \cup \delta$.
\end{proof}

The following theorem classifies $4$-gonal structures on $C$.

\begin{theorem}[Mumford, cf. \cite{acgh}, p. 193]
\label{thm:mumford}
Assume that $W_4^1(C)\neq \emptyset$. Then one of the following holds: 
\begin{itemize}
\item[(0)] \hskip 0.2cm
$\dim W_4^1(C) =0$: then $C$ has a finite number of $4$-gonal
structures;
\item[(1)]\hskip 0.2cm 
$\dim W_4^1(C) =1$:
\begin{itemize}
\item[(a)]\hskip 0.2cm  $C$ is smooth plane curve of genus $6$
and degree $5$ and the corresponding line bundle $L= H-p$, $p\in C$, where $H$ 
is a hyperplane section;
\item[(b)]\hskip 0.2cm  $C$ is bielliptic and $L$ is obtained from a bielliptic
involution $C\to E\to \P^1$, where the second map arises from a reflection 
with  respect to some point on $E$;
\item[(c)]\hskip 0.2cm  $C$ is trigonal with a unique trigonal structure and
$L = H +c$, where $c\in C(k)$, and $H$ defines a trigonal map onto $\P^1$;
\end{itemize}
\item[(2)]\hskip 0.2cm 
$\dim W_4^1(C) =2$: then $C$ is hyperelliptic and $L = 2H$, where
$H$ is a hyperelliptic bundle, or $L=H+ c +c'$, where $c,c'\in C(k)$ and $c'\neq \sigma(c)$,  
for the hyperelliptic involution $\sigma$;
in the former case, $\dim {\rm H}^0(C,2H)=3$ and any other $L$ as above defines the same
projection $H$.
\end{itemize}
\end{theorem}

Mumford's theorem holds over arbitrary ground fields.
We use it to describe explicitly abelian subvarieties
of $j(C^{(2)} \times C^{(2)})\subset J$, for $\mathsf g(C) > 4$.

\begin{theorem} 
\label{thm:ah}
Let $C$ be a curve of genus $\mathsf g(C)> 4$ and $J$ its Jacobian. 
Assume that $j(C^{(2)} \times C^{(2)})\subset J$
contains translations of abelian subvarieties.
Then one of the following holds.
\begin{enumerate}
\item $C$ is hyperelliptic with a map $C\to E$ of degree $4$ and
$$
E\subset W_4 = j(C^{(2)} \times C^{(2)}).
$$
\item $C$ is hyperelliptic with a map $C\to E$ of degree $3$ and
$E\times C \subset W_4(C)\subset J$, with the map
corresponding to the summation of cycles. Further,  
$E\subset W_3(C)\subset J$, with the 
embedding induced by the projection from $C$.
\item $C$ is bielliptic and 
there is a finite number of bielliptic structures $e_i :C\to E_i$.
Each such map defines an embedding $i_2 : E_i\subset C^{(2)}$. 
The abelian subvarieties are:
\begin{itemize}
\item $j(E_i\times E_i)= E_i\subset j(C^{(2)} \times C^{(2)})$ (of dimension
one), 
\item 
$j(E_i\times E_{i'})\subset j (C^{(2)} \times C^{(2)})$, with $i\neq i'$, (of dimension two), and
\item $j(E_i\times C^{(2)})$ - two-dimensional families of elliptic curves.
\end{itemize}
\item $C$ admits a map $h :C\to \tilde{C}$ of degree  two onto a curve of genus $2$, and
$A= \tilde{C}^{(2)}$.
\end{enumerate}
\end{theorem}

\begin{corollary}
\label{coro:D} 
There is a divisor $D\subset C^{(2)}$ such that
abelian subvarieties in $j(C^{(2)}\times C^{(2)})$ are contained
in $j(D\times C^{(2)}\cup C^{(2)}\times D)$.
In the hyperelliptic cases 1 and 2, the components of $D$ are defined by
the curves of two-cycles of 
degree $4$ and $3$ maps.
In Case 3, a bielliptic structure defines an elliptic curve in $C^{(2)}$
which is a component of $D$. In Case 4, 
a component is given by $\tilde{C}$ of genus
$2$ embedded into $C^{(2)}$. 
\end{corollary}

The remainder of this Appendix is devoted to a proof of Theorem~\ref{thm:ah}.
We use the results and techniques from \cite{AH}.

\

First we consider the case when $C$ is hyperelliptic.
Then $j(C^{(2)} \times C^{(2)})\subset J$
coincides with $W_{4}(C)$, the image of $C^{(4)}$.

\begin{lemma} 
Let $C$ be a hyperelliptic curve of genus $\mathsf g(C) > 4$.
If $W_4(C)$ contains a nontrivial abelian subvariety $A$ then:
\begin{enumerate}
\item $\dim A=1$ and there is a surjective map 
$f : C\to A$ of degree $3$ or $4$, and the embedding of $A$ into 
$W_3(C)$, resp. $W_4(C)$, is induces by this map. 
\item There is a surjective map $f:C\to \tilde{C}$ of degree $2$, 
where $\mathsf g(\tilde{C})=2$, and $A= \tilde{C}^{(2)}$.
\end{enumerate}
\end{lemma}

\begin{proof}
By Theorem 4 of \cite{AH}, if $A\subset W_3(C)$ is an abelian subvariety
and $\mathsf g(C) >4$ then $A$ is an elliptic curve and $C$ admits a map $C\ra A$ 
of degree $\le 3$, which defines the embedding of $A$ into $W_3(C)$.

Let $A\subset W_4(C)$ be an abelian variety. Let $A_2=A+(a)\subset W_{8}(C)$ be 
the translation of $A$ by an $a\in A(k)$. Each $\alpha \in A_2(k)$ defines
a line bundle $L_{\alpha}$ on $C$ (of degree 8). 
By Lemma 1 in \cite{AH},
$$
\dim {\rm H}^0(C, L_{\alpha})> \dim A+1. 
$$
By Lemma 2 in \cite{AH}, if $\dim {\rm H}^0(C, L_{\alpha})=2$ then 
$A$ has dimension 1 and there is a map $C\ra A$ of degree 4. 

Now assume that $\dim {\rm H}^0(C, L_{\alpha})\ge 3$, for all $\alpha$.  
Then $8\le 2\mathsf g(C)-2$. Since $C$ is hyperelliptic 
the map 
$$
\phi_{\alpha}\, :\, C\ra \P^{r(\alpha)}, 
$$
where $r(\alpha)=\dim  {\rm H}^0(C, L_{\alpha}) -1$, is not birational onto its
image. Thus $C$ satisfies the assumptions of Lemma 3 \cite{AH} and we conclude
that either $A\subset W_{\tilde{d}}^1(C)$, where $2\le \tilde{d}\le 4$, or there is a  
nontrivial factorization
$$
C\stackrel{\rho}{\longrightarrow} \tilde{C}\ra \P^{r(\alpha)}
$$
and an embedding $A\subset W_{\tilde{d}}(\tilde{C})$, with $\tilde{d}=4/\deg(\rho)$.
Since $C$ is hyperelliptic, 
$$
W_4^1(C)=W_3^1(C)=W_2(C).
$$
We now apply Theorem 3 \cite{AH}: if $A\subset W_2(C)$ and $\mathsf g(C)\ge 3$ 
then $C$ is not hyperelliptic, 
contradiction to our assumption. 
If $\tilde{d}=1$ then $A\subset W_1(\tilde{C})$ and hence $\tilde{C}=A$ is an elliptic curve
and we are in Case 1. 

If $\tilde{d}=2$, then 
$A\subset W_2(\tilde{C})$ and either 
$1\le \mathsf g(\tilde{C})\le 2$ or $\tilde{C}$ is bielliptic.
If $\tilde{C}$ were bielliptic we would get a degree 4 map from $C$ onto
an elliptic curve. If $\mathsf g(\tilde{C})=1$ then $C$ is bielliptic, contradicting 
the assumption on the genus of $C$. 
If $\mathsf g(\tilde{C})=2$ then $A=W_2(\tilde{C})$. 
\end{proof}

From now on we assume that $C$ is nonhyperelliptic.
In particular, $C^{(2)}$ does not contain rational curves.

Assume that the map $j : C^{(2)} \times C^{(2)}\to J$
is an embedding outside of the diagonal $\Delta \cup  \delta$.
Then $A\subset  j(C^{(2)} \times C^{(2)})$
lifts birationally to $C^{(2)}\times C^{(2)}$
and the image of $A$ in each projection to $C^{(2)}$ 
must be an elliptic curve or a point. 
It follows that $A$ is contained in a product
of elliptic curves in $C^{(2)} \times C^{(2)}$. 
It remains to apply the following lemma (see, e.g., \cite{AH}).

\begin{lemma} 
Assume that $C$ is a nonhyperelliptic curve of genus $\mathsf g(C) > 4$.
Let $\tilde{C}\subset C^{(2)}$ be  a curve of genus
$1$ or $2$. Then there is a degree 2 map $C\to \tilde{C}$.
\end{lemma}

Now we assume that there is a nontrivial subvariety 
$Z\subset C^{(2)} \times C^{(2)}$ not contained in the diagonal 
$\Delta(C^{(2)})$ such that $j: Z\to J$ is not an embedding 
on the complement of $\Delta(C^{(2)})$. 
This occurs only when $C$ is one of the curves
satisfying Mumford's theorem \ref{thm:mumford}. 
In particular, $W_4^1(C)\neq \emptyset$. 

A map $f: C\to \P^1$ of degree $4$ determines a possibly reducible curve
$$
C_f:=\left(C\times_{\P^1}C \setminus \Delta(C)\right)/\Z/2,
$$
where the $\Z/2$-action interchanges the factors. 
The curve $C_f$  parametrizes unordered pairs of points $(c,c')$ in the  fibers of $f$. 
We have a natural embedding $\xi_f \,:\, C_f\subset C^{(2)}$ and 
a projection $\eta_f\, :\, C_f\to \P^1$ of degree $6$. 
In addition, there is
a natural nontrivial fiberwise involution 
$\iota : C_f \to C_f$  which maps
the degree 2 cycle $(c,c')$ into the complementary cycle in the same fiber of $f$, i.e., 
if $(c,c',c'',c''')$ is a fiber of $f$ then $\iota (c,c') = (c'',c''')$.
The map $g :C_f/\iota = C_g\to \P^1$
has degree $3$. The curve $C_g$ parametrizes the splittings of the fibers
of $f$ into a pairs of degree $2$ zero-cycles.

Let $n_2,n_{2,2},n_3, n_4$ be the number of different
ramification types for $f: C\to \P^1$, i.e.,
$n_i$  is the number of fibers of $f$ with one
ramification of multiplicity $i$, for $i=2,3,4$,
and $n_{2,2}$ the number of fibers with 2 simple ramifications.

\begin{lemma} 
\label{lemm:f-hurw}
Assume one of the following holds:
\begin{itemize}
\item $C_f$ is irreducible and at least one of the $n_2,n_3$ or $n_4$ is nonzero or
\item $C_f$ is reducible and $f\,:\, C\ra \P^1$ is not a Galois covering.
\end{itemize}
Then 
$$
\mathsf g(C_g)\geq \frac{1}{2} \mathsf g(C) \geq 3.
$$
\end{lemma}

\begin{proof}
The Galois group ${\rm Gal}(f)$ of the $4$-covering $f:C\to \P^1$
is one of the following
$$
\mathfrak S_4, \mathfrak A_4, \mathfrak D_4, \Z/4, \Z/2\oplus \Z/2.
$$
The curve $C_f$ is irreducible iff ${\rm Gal}(f)$
acts transitively on the subsets of $6$ unordered pairs
of $4$ points. This is the case of $\mathfrak S_4,\mathfrak A_4$.
In the case of $\mathfrak D_4$ and $\Z/4$ the set of $6$ unordered pairs
splits into two orbits of orders $4$ and $2$, respectively, and
in the case of $\Z_2 \oplus \Z/2$ there are orbits of order
$2$. This determines all possible splittings of $C_f$.

Note that for $\mathfrak D_4, \Z/4$ the curve $C$ is
isomorphic to a connected component of $C_f$ and there is
a natural involution $\theta$ on $C$ given by a central element
of order $2$ in $\mathfrak D_4$ and $\Z/4$, respectively.
The quotient hyperelliptic curve
$C/\theta$ coincides with the second component of $C_f$, so that $C_f=C\cup C/\theta$.

It is easy to obtain numerical characteristics of $C_f$ and $C_g$
in all cases. By Hurwitz' formula, 
\begin{equation}
\label{eqn:neww}
2\mathsf g(C)- 2= n_2 + 2 n_{2,2} + 2n_3 + 3n_4 - 8.
\end{equation}
There is a natural correspondence between the ramification diagrams
for the fibers of $C, C_g$ and $C_f$. 
Write $(i_1,\ldots, i_r)$ for the ramification type of a fiber 
of the corresponding map to $\P^1$, a smooth fiber of a map of degree $r$ 
corresponds to $(1,\ldots, 1)$, $r$-times.
We have
$$
\begin{array}{|c|c|c|}
\hline
         C& C_g    & C_f \\
         \hline          
(1,1,1,1) & (1,1,1) & (1,1,1,1,1,1)\\
(2,1,1)   & (2,1)   & (2,2,1,1)\\
(2,2)     & (2,1)   & (4,1,1)\\
(3,1)     & (3)     & (3,3)\\
 (4)      &  (3)    & (6)\\
\hline
\end{array}
$$
Note that the action of $\Z/2$ on the fibers of $C_f$ is free
only in the first and third cases.
We obtain 
$$
\mathsf g(C_g) = \frac{1}{2} n_2+ \frac{1}{2}n_{2,2} +n_3 + n_4 - 2.
$$
Using ~\eqref{eqn:neww},
$$
2\mathsf g(C_g) = n_2+ n_{2,2} + 2n_3 + 2n_4 - 4 > \frac{1}{2} n_2 +n_{2,2} + n_3 +
\frac{3}{2}n_4 - 3 = \mathsf g(C)
$$
and 
$$
\frac{1}{2} n_2 + n_3 + \frac{1}{2}n_4-1= 2\mathsf g(C_g)- \mathsf g(C)
$$
and hence $\mathsf g(C_g) \geq 3$, if either $\mathsf g(C) > 5$ or some singular
fibers of $f$  are not  of type $(2,2)$.
Note that if the action of $\Z/2$ is free on $C_f$ then $n_{2,2}=0$
and $\mathsf g(C_g)=\mathsf g(C)+1 > 5$.

Consider the reducible case with the Galois group $\mathfrak D_4$.
As explained above,
$C$ is obtained as a quotient of the Galois cover
by a noncentral $\Z/2\subset \mathfrak D_4$.
Thus there is also an action of a central
$\Z/2$ on $C$. The family of degree two cycles splits
into a curve $C$ and $C/\Z/2$.
The only singular fibers are
$(2,1,1), (2,2),(4)$ and the corresponding fibers
or $(1,1), (2),(2)$, respectively.
Thus $C_g$ is a hyperelliptic curve
and $C_f$ is a double covering with ramifications over points
in $C_g$ with some of them not belonging to the invariant
points of the hyperelliptic involution.
If the Galois group is $\Z/4$ then $C_f$ is a double covering
of $C_g$, doubly ramified over ramification points
of $C_g\ra \P^1$. This gives a lower bound
on the genus of $C_g$: 
$$
\mathsf g(C)= \frac{3}{2} n_4 - 5 \,\, \text{  and } \,\, \mathsf g(C_g) = n_4 - 1.
$$ 
Thus 
$$
2\mathsf g(C_g) > \mathsf g(C).
$$
When  $\Gal(f)=\Z/2 \oplus \Z/2$  the curve $C$ is actually a fiber product
of two hyperelliptic curves over $\P^1$.
\end{proof}

We continue the investigation of $Z$.  
Consider the map 
$$
C_f\times C_f \stackrel{\xi_f^2}{\longrightarrow} C^{(2)} \times C^{(2)} 
\stackrel{j}{\longrightarrow} J.
$$
There is natural action of $\mathfrak D_4$ on $C_f\times C_f$
which contains the permutation of fibers $i$ and
fiberwise involutions $(\iota, id), (id, \iota) $ on
$C_f\times C_f$.

\begin{lemma}
Let $f :C\to \P^1$ be a map of degree $4$.
The restriction of $j$ to the image $\xi_f^2(C_f\times C_f)$ is the 
composition of the quotient by the involution
$i\circ (\iota,\iota)$, which is conjugate to $i$,
with a birational embedding, which is an isomorphism onto its image
on the complement to the diagonal $\Delta(C^{(2)})$ and 
$$
\bigcup_{f'\neq f} \xi_{f'}^2(C_{f'}\times C_{f'})\subset C^{(2)} \times C^{(2)},
$$
over all other 4-gonal maps $f'$. In particular, we have a factorization

\centerline{
\xymatrix{
 C_f\times C_f        \ar[d]   \ar[r]^{\xi_f^2} & C^{(2)} \times C^{(2)} \ar[d]_j \\
 C_f^{(2)}           \ar[r]_{j_f}               &   J.
}
} 

\end{lemma}

\begin{proof} 
Consider the pair of cycles $(x_1,x_2), (y_1,y_2)\in C_f\times C_f$. Then
$$
(x_1+x_2)-\iota(y_1,y_2)=y_1+y_2- \iota(x_1,x_2)
$$ 
which implies
that $j$ factors through the quotient by 
$i\circ (\iota,\iota)$
and hence $C_f\times C_f\subset Z$.
Assume that 
$$
(x_1+x_2)-(z_1+z_2)=(y_1+y_2) -(w_1+w_2),
$$
where
$(z_1+z_2)\neq \iota(y_1,y_2)$. Since $C$ is nonhyperelliptic,
$$
(z_1+z_2)\neq (x_2+x_3),(x_2,x_3)= \iota(x_1,x_2)
$$
and we obtain 
$$
y_1+y_2+z_1+z_2 = x_1+x_2 + w_1 + w_2
$$ 
which corresponds to a different $4$-gonal structure $f'$ on $C$
with $(x_1,x_2),(y_1,y_2)\subset C_{f'}\times C_{f'}$.
\end{proof}

\begin{corollary}
If $C_f$ is irreducible then for any other $C_{f'}$ the intersection
$$
\xi_f^2(C_f\times C_f) \cap \xi_{f'}^2(C_{f'}\times C_{f'})
$$
is finite. 
Hence the image $j \circ \xi_f^2(C_f\times C_f)$ is birational
to $C_f^{(2)}$. 
\end{corollary}

\begin{proof}
Mumford's theorem~\ref{thm:mumford} implies that there is at most
a one-dimensional family of $4$-gonal structure on $C$. Hence
the intersection of $\xi_f^2(C_f\times C_f)$ with the union of
 $\xi_{f'}^2(C_{f'}\times C_{f'})$, over all other  
$4$-gonal structures $f'$, is at most a curve in $\xi_f^2(C_f\times C_f)$.
\end{proof}

\begin{lemma}
Let $Z^{\circ}\subset C^{(2)}\times C^{(2)}\setminus \Delta(C^{(2)})$ be
the maximal subvariety such that 
$j$ restricted to $Z^{\circ}$ is not an isomorphism onto its image. 
Let $Z$ be the Zariski closure of $Z^{\circ}$ in $C^{(2)}\times C^{(2)}$. Then
$$
Z=\cup_f Z_f,
$$
over the set of 4-gonal structures on $C$. Here $Z_f=C_f\times C_f$. 
\end{lemma}

\begin{proof} 
Assume that 
$$
j((x,x_1),(y,y_1)) = j((y_2,y_3),(x_2,x_3)).
$$ 
Then
$$
(x+x_1)-(y+y_1) = (y_2+y_3)-(x_2+x_3)
$$ 
and hence
$$
(x+x_1)+(x_2+x_3) = (y+y_1)+(y_2+y_3)
$$ 
which means that for
any $z,w\in C^{(2)}\times C^{(2)}, z\neq w$ with $j(z)=j(w)$
there is a  $4$-gonal map $f$ so that  $z,w\in C_f\times C_f$
and $i(\iota,\iota)(z) = w$.
\end{proof}

\begin{corollary}
If $C$ is not hyperelliptic then
$j \,:\, C^{(2)}\times C^{(2)}\ra J$ is a birational isomorphism onto its image.
\end{corollary}

\begin{proof}
Indeed, Mumford's theorem~\ref{thm:mumford} implies that
the family of $4$-gonal maps on $C$ is at most one-dimensional. Hence
$Z_f$ has dimension $2$ for any $4$-gonal structure.
Since we have at most $W_4^1(C)$ in the nonhyperelliptic
case, $Z$ is at most a one-dimensional family of surfaces
and the map $j$ is an embedding outside of $Z$ and $\delta$.
Thus $j$ is a birational isomorphism onto its image.
\end{proof}

\begin{lemma}
\label{lemm:ain-z}
Any abelian subvariety $A\in j(C^{(2)}\times C^{(2)})$ which is not
in $j(Z)$ is contained in either $j(E_1\times C^{(2)})$ or
$j(C^{(2)}\times E_1)$, where $E_1\subset C^{(2)}$ is an elliptic
curve. This embedding corresponds to a bielliptic structure
$h_1: C\to E_1$.
\end{lemma}

\begin{proof}
Let $A\in j(C^{(2)}\times C^{(2)})\setminus j(Z)$.
Then $j^{-1}(A)=\tilde{A}$ is birationally isomorphic to $A$.
The projections $\pi_1,\pi_2 :\tilde{A}\to C^{(2)}$ map it either into
an elliptic curve $E_i$ or into a point.
Indeed, by assumption on $C$ the surface $C^{(2)}$ is of general type
and does not contain a rational curve since $C$ is not hyperelliptic.
Thus the image $\pi_i(A)$, $i =1,2$, is an elliptic curve for
at least one $i=1,2$, for example $\pi_1$.
By Theorem~\ref{thm:ah}, $C$  is bielliptic
with an embedding $E_i\hookrightarrow C^{(2)}$ corresponding to the
bielliptic structure $h_i : C\to E_i$.
Thus $A\subset E_1\times C^{(2)}$.
\end{proof}

\begin{corollary}
If $A\subset j (C^{(2)}\times C^{(2)})$
does not correspond to Case $3$ of Theorem~\ref{thm:ah} then
$A\subset j(Z)$.
\end{corollary}

Let us consider individual subvarieties $Z_f\subset Z$.

\begin{lemma} 
\label{lemm:not-bi}
Let $C$ be nonhyperelliptic, of genus  $\mathsf g(C) > 4$. 
Then the irreducible curve $C_f$ is not bielliptic.
\end{lemma}

\begin{proof}
Assume that  $C_f$ is bielliptic. Then
$C_g$ is also bielliptic.
Indeed, consider the bielliptic map $C_f\to E$.
There is a degree $4$ surjective map $C_f^{(2)}\to C_g^{(2)}$
which maps an elliptic curve $E\subset C_f^{(2)}$ to an elliptic
or a rational curve. In the second case, the involution on $C_f$ maps
to a hyperelliptic involution on $C_g$.

\centerline{
\xymatrix{
C_f \ar[r]\ar[d]  &  E\ar[d]  \\
C_g \ar[r]        & \P ^1 
}
}

If $C_f, C_g$ are irreducible then $C_g$ is trigonal. Hence
$\mathsf g(C_g)\leq 2$. However, by Lemma~\ref{lemm:f-hurw}, 
$\mathsf g(C_g)\geq 3$, contradiction.
Thus the image of $E$ is an elliptic curve and
the diagram is:

\centerline{
\xymatrix{
C_f \ar[r] \ar[d] & E\ar[d]  \\
C_g \ar[r]        & \tilde{E}=E/\Z/2,
}
} 

\noindent
where $E\to \tilde{E}$ is an unramified covering of degree $2$.
Thus $C_g$  is bielliptic and trigonal and $C_f\ra C_g$ is an 
unramified double cover induced
by $E\to \tilde{E}$.

This implies that $\mathsf g(C_g) \geq 4$. Indeed, if the trigonal
structure $C_g\to \P^1$ were invariant with respect to the bielliptic
involution $\iota$ then the latter would induce an involution
$\iota'$ on $\P^1$. Since $\iota'$ has exactly two invariant points
all the invariant points of $\iota$ are contained in two
fibers of the map $C_g\ra \P^1$. 
Thus $C_g$ is a degree 2 covering of $\tilde{E}$ ramified in at most
$6$ points which implies the result.
If on the other hand, $C_g\ra \P^1$ is not $\iota$-invariant then $C$ has $3$
different trigonal structures and hence maps birationally into
a curve of bi-degree $(3,3)$ in $\P^1\times \P^1$.
Then 
$$
2\mathsf g(C_g)-2 \leq 3(H_1+H_2)(H_1+H_2)= D(D+K) = 6
$$
and hence $\mathsf g\leq 4$.
Note that the genus of $C$ is equal to $\mathsf g(C_g)-1$
which follows from \eqref{eqn:neww}
because of the absence of $2,2$ and $4$ fibers in this case.
This finishes the proof of the lemma.
\end{proof}

\begin{remark} 
The construction above is part of the classical
Prym variety construction with $J= {\rm Prym}(C_g,\theta)$,
where $\theta$ is a point of order two defining the
nonramified covering $C_f$. Thus 
$$
\mathsf g(C)= \dim {\rm Prym}(C_g,\theta)= \mathsf g(C_g)-1.
$$
\end{remark}

\begin{lemma}
If $C$ is not a plane curve of degree $5$ then
any abelian subvariety in the image $j(C^{(2)}\times C^{(2)})\subset J$ 
is contained in $\bigcup_f Z_f$, over all $f$ defining 4-gonal structures on $C$, 
with reducible $C_f$. 
\end{lemma}

\begin{proof} 
If $C_f$ is irreducible then $C^{(2)}$ does not contain
an elliptic curve by Lemma~\ref{lemm:not-bi}. 
The set where the map is not bijective
coincides with the intersection locus with other $C_{f'}\times C_{f'}$.
Since under the assumption of the lemma
there is only a finite number of $4$-gonal structures
which don't correspond to a bielliptic structure,
the one-dimensional families of $C_{f'}$ correspond to
bielliptic maps. They define the only curves in $C^{(2)}$
where 
$$
j_f\,:\, C^{(2)}_f\setminus \Delta(C_f) \to J
$$
is not an isomorphism.
Thus only reducible $C_{f}$ contribute abelian subvarieties in the image.
\end{proof}

\begin{lemma} 
\label{lemm:123}
If $C_f$ is reducible
and $j(Z_f)$ contains an abelian subvariety then:
\begin{enumerate}
\item $C$ is bielliptic and $f:C\to \P^1$
is a composition $C\to E$ with an involution on $E$.
Then $C^{(2)}_f$ contains $E$.
\item The map $f:C\to \P^1$ is a composition of
a degree two map $C\to \tilde{C}$ onto a curve of genus two and
a hyperelliptic projection $\tilde{C}\to \P^1$.
In this case $j(Z_f)$ contains the abelian surface $j(\tilde C^{(2)})$.
\item If $\Z/2\oplus \Z/2$ acts on $C$ then we get
a combination of the above two cases, depending on the genus
of the curves $C_i$.
\end{enumerate}
\end{lemma}

\begin{proof} 
It is evident that in all of these cases
there is an abelian subvariety in $j(Z_f)$.
In cases $\mathfrak D_4, \Z/4$ 
the image of $j(Z_f)$  is a birational embedding for
three surfaces
$$
C^{(2)}, (C/\theta)^{(2)}, C\times (C/\theta)
$$ 
into $J$.
Thus if $C$ is not bielliptic the abelian subvariety
may be contained in $C\times (C/\theta)$ (and then $C/\theta$ is an elliptic
curve - contradicting the assumption) or in
$(C/\theta)^{(2)}$. The latter is hyperelliptic and hence
$(C/\theta)^{(2)}$ does not contain elliptic curves if
$\mathsf g(C/\theta)\geq 3$. Thus the only possibility is
$\mathsf g(C/\theta)= 2$ and $(C/\theta)^{(2)}$ is birational to an abelian
surface.

Similarly, in the case of 
$\Z/2\oplus \Z/2$ we have a union of $C_i\times C_{i'},i\neq i'$
and $C_i^{(2)}$ which implies the result.
\end{proof}

Thus we have shown that unless $C$ is a smooth plane curve of degree $5$ and
genus $6$ the abelian subvariety is contained in $j(Z)$ only in the cases
described by the theorem.

\begin{lemma}
Let $C\subset \P^2$ be a smooth plane curve of degree $5$. Fix a point
$c\in C$. Let $f=f_c$ be the 4-gonal structure on $C$ defined by projection from
$c$. Then
\begin{enumerate}
\item $Z$ is irreducible,
\item $Z$ contains an open subvariety
$Z_0\subset Z$  such that any point in $Z_0$ is contained in
a unique $Z_{f_c}$,
\item there is a rational map $\pi_f : Z\to C$ with
the closures of a fiber 
$\pi_f^{-1}(c) = C_{f_c}\times C_{f_c},c\in C$,
\item the map $j$ on $Z_0$ has degree $2$,
\item the map $j$ on $Z_0$ commutes with $\pi_f$
and there is a commutative diagram of maps

\centerline{
\xymatrix{
Z\ar[r]^{\pi_f} \ar[d]_j  & C \ar@{=}[d]\\
j(Z) \ar[r] & C,
}
}

\item any abelian subvariety $A\subset j(C^{(2)}\times C^{(2)})$
is contained in $j(Z_{f_c})$ for some $c\in C$.
\end{enumerate}
\end{lemma}

\begin{proof}
We have already proved that any abelian subvariety in 
$j(C^{(2)}\times C^{(2)})$ is contained
in $j(Z)= \bigcup j(Z_{f_c})$ since in the case of a smooth curve
of degree $5$ any linear series in $W_4^1(C)$ corresponds to one
of the four-gonal maps $f_c,c\in C$.

First we show that $Z\subset C^{(2)}\times^{(2)}$ is irreducible.
As we have mentioned $Z$ is a union of varieties $C_{f_c}\times
C_{f_c}),c\in C$.
It suffices to show
that a generic curve $C_{f_c}$ is irreducible.
If $C_{f_c}$ is not
irreducible then by Lemma~\ref{lemm:123} there is a $\Z/2$ action on $C$, with 
center $c$ of the projection $f_c$ fixed by this action. The structure of a
smooth projective curve of degree $5$ is unique on $C$ and hence the group
$\Z/2$ above is a subgroup
of a finite subgroup $I_C\subset \PGL_3(k)$ which stabilizes $C\subset \P^2$.
Since there is only a finite number
of subgroups $\Z/2\in I_C$, the  a number of projections $f_c$ such that
$C_{f_c}$ is reducible is also finite. Hence $Z$ is irreducible.

Any two-cycle $(x,x_1)$ on $C$ defines a unique line
$\mathfrak l(x,x_1)$ with $\mathfrak l(x,x_1)\cap C$ containing $(x,x_1)$.
When $x,x_1$ are different, $\mathfrak l(x,x_1)$ is the unique line through $x,x_1$.
When $x=x_1$ then $\mathfrak l(x,x_1)= \mathfrak l(x,x)$ is the tangent to the smooth curve
$C$ at $x\in C$.
Let $((x,x_1),(y,y_1))$ be a point in $Z$.

If $\mathfrak l(x,x_1)\neq \mathfrak l(y,y_1)$ then  $\mathfrak l(x,x_1), \mathfrak l(y,y_1)$ 
intersect at a unique point
$c\in C$. This defines a rational
surjective map 
$$
\begin{array}{rcl}
\pi_f \,:\, C^{(2)}\times C^{(2)} & \ra     & C\\
            ((x,x_1),(y,y_1))      & \mapsto & c:=\mathfrak l(x,x_1)\cap \mathfrak l(y,y_1)
\end{array}
$$ 
from an open subvariety  $Z^{\circ}\subset Z$, defined by $\mathfrak l(x,x_1)\neq \mathfrak l(y,y_1)$.
If $\mathfrak l(x,x_1)=\mathfrak l(y,y_1)$ then the points $x,x_1,y,y_1$ belong to the same
$5$-cycle in the intersection $\mathfrak l(x,x_1)\cap C$.
If $x+x_1+y+y_1$ is the fiber of $f_c$ defined by the fifth point
$c$ in the intersection $\mathfrak l(x,x_1)\cap C$ then $\pi_f((x,x_1),(y,y_1))= c$
and hence $\pi_f$ is well-defined on such $((x,x_1),(y,y_1))\in Z$.

Thus $\pi_f$ may fail to be well-defined  only on the surface
$$
S_{\mathfrak l} =\{((x,x_1),(x,x_2))\subset Z, \,|\,  \mathfrak l(x,x_1)=\mathfrak l(x,x_2)\}.
$$ 
Consider the restriction of $j$ on $Z_0 = Z\setminus S_{\mathfrak l}$.
Let us show that $j$ is exactly of degree $2$ on $Z_0$
and $j(S_{\mathfrak l})\subset j(C\times C)$, where $C\subset J=J^1$
is a standard embedding.

Indeed, assume that $((x,x_1),(y,y_1))\in Z_0$. Then they define a unique
point $c$ which is either the intersection point
$\mathfrak l(x,x_1)\cap \mathfrak l(y,y_1)$ or the the fifth point
in the intersection $\mathfrak l(x,x_1) = \mathfrak l(y,y_1) \cap C$.
The equality of zero-cycles 
$$
(x+x_1)-(w +w_1)= (y+y_1)-(z+ z_1)
$$ 
implies
$$
(x+x_1) + (z+ z_1)= (y+y_1)+(w +w_1).
$$
Thus if $\mathfrak l(x,x_1)\neq \mathfrak l(y,y_1)$
then $\mathfrak l(x,x_1)= \mathfrak l(z,z_1), \mathfrak l(y,y_1)=\mathfrak l(w,w_1)$. 
Hence
$\iota_{f_c}(x,x_1) = \iota_{f_c}(y+y_1)$. We also have
$\pi_f((x,x_1),(y,y_1))= \pi_f( (w,w_1),(z,z_1))$ which yields the result
in this case.

In the second case, $(x+x_1)-(x_2+x_3) =(z_1+z_2)-(w_1+w_2)$ implies
$(x+x_1)+(w_1+w_2)= (x_2+x_3)+(z_1+z_2)$ and that
$\mathfrak l(w_1,w_2)= \mathfrak l(x_1,x_2), \mathfrak l(z_1,z_2  )=\mathfrak l(x_2+x_3)$. 
Hence both cycles are contained in $\mathfrak l(x,x_1)$.
Unless $(w_1,w_2) = (x_2,x_3)$ and $(z_1,z_2)= (x,x_1)$
there is a relation between two positive cycles from
$\mathfrak l(x,x_1)\cap C$ of degree $<4$ which cannot happen 
since $C$ is neither trigonal, nor hyperelliptic.
Thus we showed that the map $j:Z_0\to j(Z_0)$ is a map of degree $2$.
If $((x,x_1),(x,x_2))\subset S_\mathfrak l$ then 
$$
j((x,x_1),(x,x_2))= x_1-x_2\in j(C\times C).
$$
This completes the description of the map $j$ on $Z$.

Thus $j$ commutes with the map $\pi_f$ on $Z_0$ and defines a rational
surjection $\pi'_f\,:\, j(Z)\to C$.
For any $A\subset j(Z) \setminus j(C\times C)$ the image of $\pi'_f(j(A))$
is a point.
Hence any such $A$ is contained the closure of the fiber  of $\pi'_f$
which is equal to $j(Z_f)$.
On the other hand, $j(C\times C)$ does not contain any abelian subvariety
and hence the above statement holds for any abelian subvariety in
$j(Z)$ or equivalently in $j(C^{(2)}\times C^{(2)})$.
\end{proof}

\begin{corollary}
If $C$ is smooth plane curve of degree $5$ then
any abelian subvariety of $j(Z)$ is contained in one
of the $j(Z_f)$. Moreover, a maximal abelian subvariety
$A\subset j(Z)\subset j(C^{(2)}\times C^{(2)})$ is
necessarily of dimension 2 and it exists
only if there as an action of $\Z/2$ on $C$. In this case
$\mathsf g(C/\Z/2)=2$ and $A= J(C/\Z/2)$.
There is exactly one such subvariety for any
subgroup $\Z/2$ in the group of automorphisms of $C$.
\end{corollary}

\begin{proof}
The number of ramification points is at most $6$ since they
are contained in a union of a $\Z/2$-invariant point and a line in $\P^2$.
By the genus calculation, $\mathsf g(C/\Z/2) = 2$. This curve defines an
abelian surface
in $J$.
\end{proof}

This completes the proof of the main theorem apart from the
case of a trigonal curve $C$.

\begin{lemma} 
Let $C$ be a smooth projective curve of genus 
$\mathsf g(C)>4$. If $C$ has a trigonal structure then it is unique. 
\end{lemma}

\begin{proof}
Two distinct structure would give a rational embedding
of bidegree $(3,3)$ into $\P^1\times \P^1$.
The genus computation shows 
$2\mathsf g_a(C)-2= (3,3)\cdot (1,1)= 6$, where $\mathsf g_a(C)$ is the
arithmetic genus of the image of $C$.  
Hence $\mathsf g_a(C)=4$ and $\mathsf g(C)\leq \mathsf g_a(C)$,
contradiction.
\end{proof}

\begin{lemma} 
Assume that $C$ is trigonal curve of genus $\mathsf g > 4 $, with a projection
$h: C\to \P^1$ of degree $3$. 
Then $j(C^{(2)}\times C^{(2)})\subset J$
contains no abelian subvarieties.
\end{lemma}

\begin{proof}
First of all, $C$ is neither bielliptic, nor hyperelliptic, nor
a plane curve of degree 5 (cf. \cite{AH}). Moreover, the trigonal structure is unique
(see \cite{acgh}).
Each point $c\in C$ defines a degenerated 4-gonal structure. 
Apart from these, there are only finitely many other $4$-gonal structures.
Since the trigonal structure on $C$ is unique
there is no $\Z/2$-action on $C$, by the same genus estimate 
as in Lemma~\ref{lemm:not-bi}. 
Hence for any additional $4$-gonal structure on $C$ the curve
$C_f$ is irreducible.

We have two maps
$$
\begin{array}{rcc}
\pi_c\,:\, C & \ra & C^{(2)}\\
           c'& \mapsto & c'+c
\end{array}
$$
where $c\in C(k)$ and 
$$
\begin{array}{rcc}
\chi\,:\, C & \ra & C^{(2)}\\
           c'& \mapsto & c''+c''', 
\end{array}
$$
where $c', c'', c'''$ are in the same fiber of the projection defining 
the trigonal structure. 

If follows that 
$$
Z=\cup_f C_f\times C_f \bigcup \cup_{c\in C(k)} (\pi_c(C)\cup \chi(C))^2\subset C^{(2)}\times C^{(2)},
$$
where $f$ runs over a finite set of nontrivial 4-gonal structures.

Thus 
$$
j(Z)=\cup j(C\times C) \cup j(C^{(2)}\times \chi(C))\bigcup \cup_f j(Z_f), 
$$ 
where $f$ runs over (a finite set of) nondegenerate $4$-gonal structures on $C$. 

Note that $j(C\times C)$ and $j(Z_f)$ do not
contain abelian subvarieties. 
Consider  $(x,y)\times (c,c'), (z,w)\times (s,s')\in C^{(2)}\times \chi(C)$. 
Then
$$
j ((x,y),(c,c'))= j((z,w,),(s,s'))
$$
if
$$
(x+y) - (c +c') = (z+w) -(s+s')
$$
which is equivalent to $x+y + c'= z+w +s'$. The uniqueness
of the trigonal structure implies that $(x,y) = (c,c')$
and $(z,w)=(s,s')$. Hence the additional gluing
in $j(C^{2}\times \chi(C))$ occurs only
on the diagonal in $\chi(C)\times \chi(C)$. 

Thus there are no abelian subvarieties in $j(C^{(2)}\times C^{(2)})$.
\end{proof}

%
%
\bibliographystyle{smfalpha}
\bibliography{torelli}
%


\end{document}